\documentclass[10pt]{amsart}

\usepackage{times}
\usepackage[dvipsnames]{xcolor}
\usepackage{color}		
\usepackage{graphicx}	

\usepackage{amsmath,amssymb,mathrsfs,bm, enumerate}
\usepackage{mathtools}

\usepackage{yhmath} 

\usepackage[colorlinks=true, linkcolor=black, citecolor=black, urlcolor=black]{hyperref}
\usepackage{tikz-cd} 
\usetikzlibrary{decorations.pathreplacing,calligraphy}
\usepackage[all,cmtip]{xy}
\xyoption{arrow}

\usepackage{stmaryrd}
\usepackage{tcolorbox}
\usepackage{graphicx}
\usepackage{amsmath,amsthm,amsfonts,amssymb}
\usepackage[mathscr]{euscript}

\usepackage{xargs}  

\usepackage[letterpaper, top=2.9cm, bottom=2.9cm, left=3cm, right=3cm, heightrounded,bindingoffset=0mm]{geometry}

\usepackage{enumitem}

\newcommand{\ZZ}{\mathbb{Z}}
\newcommand{\QQ}{\mathbb{Q}}

\newcommand{\NN}{\mathbb{N}}
\newcommand{\CC}{\mathbb{C}}

\newcommand{\PP}{\mathbb{P}}
\newcommand{\VV}{\mathbb{V}}

\newcommand{\bMgn}[1][1]{{\overline{\mathcal{M}}_{g,#1}}}
\newcommand{\bMgN}{{\overline{\mathcal{M}}_{g,n}}}
\newcommand{\tMgn}[1][1]{{\widetriangle{\mathcal{M}}_{g,#1}}}

\newcommand{\bMon}[1][1]{{\overline{\mathcal{M}}_{0,#1}}}

\newcommand{\tMgN}{{\widetriangle{\mathcal{M}}_{g,n}}}

\newcommand{\Oo}{\mathscr{O}}

\newcommand{\Gm}{\mathbb{G}_m}

\newcommand{\Ll}{\mathcal{L}}
\newcommand{\Op}{{\mathcal{O}}}

\newcommand{\Qp}{{Q_{+}}}
\newcommand{\Qm}{{Q_{-}}}
\newcommand{\Qpm}{{Q_{\pm}}}

\newcommand{\Aut}{\mathrm{Aut}}
\newcommand{\End}{\mathrm{End}}

\newcommand{\Vir}{\mathrm{Vir}}
\newcommand{\Der}{\mathrm{Der}}

\newcommand{\rank}{{\rm rank}}

\usepackage{cleveref}

\theoremstyle{plain}
\newtheorem{thm}{Theorem}

\newtheorem{corx}{Corollary}

\newtheorem{lemma}[subsubsection]{Lemma}
\newtheorem{proposition}[subsubsection]{Proposition}

\theoremstyle{remark}
\newtheorem{question}{Question}
\newtheorem{example}[subsubsection]{Example}
\newtheorem{remark}[subsubsection]{Remark}

\theoremstyle{definition}
\newtheorem{definition}[subsubsection]{Definition}

\usetikzlibrary{decorations.pathreplacing,calligraphy}
\definecolor{ao(english)}{rgb}{0.0, 0.5, 0.0}

\begin{document}

\title[On global generation of bundles on moduli spaces of curves from representations of VOAs]{On global generation of vector bundles on the moduli space of curves\\ from representations of vertex operator algebras}
{\tiny{\author[C.~Damiolini]{C.~Damiolini}
\address{Chiara Damiolini  \newline \indent   Department of Mathematics, Rutgers University,  Piscataway, NJ 08854  \newline  \indent Department of Mathematics, University of Pennsylvania, Philadelphia, PA 19104-6395}
\email{chiara.damiolini@gmail.com} }}
{\tiny{\author[A.~Gibney]{A.~Gibney}
\address{Angela Gibney  \newline \indent   Department of Mathematics, Rutgers University,  Piscataway, NJ 08854    \newline \indent Department of Mathematics, University of Pennsylvania, Philadelphia, PA 19104-6395}
\email{angela.gibney@gmail.com} }}
 
\thanks{AG was partially supported by NSF DMS-1902237.}

\begin{abstract} 
We consider global generation of sheaves of coinvariants on moduli  of curves given by simple modules over certain vertex operator algebras, extending results for affine VOAs at integrable levels on stable pointed rational curves.  A number of examples illustrate the subtlety of the problem.
\end{abstract}

\maketitle

\section{Introduction}
Given an object in any category, a natural objective is to find the maps admitted by it.  On  $\overline{\mathcal{M}}_{g,n}$, the moduli stack parametrizing families of stable $n$-pointed curves of genus $g$, globally generated coherent sheaves define rational maps, which  are regular on the locus where they are free.

Sheaves of coinvariants, determined by $n$ simple admissible modules $W^i$ over a vertex operator algebra $V$ (a VOA), are defined on $\overline{\mathcal{J}}_{g,n}$, the moduli stack parametrizing families of stable pointed curves with first order tangent data. Under mild assumptions they descend to sheaves  $\mathbb{V}(V, W^\bullet)$  on  $\overline{\mathcal{M}}_{g,n}$  \cite{dgt1, dgt2}. If $V$ is $C_2$-cofinite, these sheaves are coherent \cite{DGK}, if $V$ is also rational, they are vector bundles \cite{dgt2}, and if strongly rational,  their Chern classes are tautological  \cite{dgt3}  (see \S \ref{sec:Background} for definitions). Examples include those given by affine VOAs, certain $W$-algebras, even lattice VOAs, and holomorphic VOAs (like the moonshine module), and others obtained as tensor products, orbifold algebras, and through coset constructions. 

Affine VOAs are derived from (quotients of) the affinization of a Lie algebra $\mathfrak{g}$, and  $\ell \in \mathbb{C}$, with $-\ell$ not equal to the dual Coxeter number.  The simple affine VOA $L_{\ell}(\mathfrak{g})$, 
generated by its degree 1 component $\mathfrak{g}$, is strongly rational if and only if  
$\ell \in \mathbb{Z}_{>0}$.  For $\mathfrak{g}$ reductive,  $\mathbb{V}(L_{\ell}(\mathfrak{g}), W^\bullet)$ was shown to be a vector bundle on $\overline{\mathcal{M}}_{g,n}$ in \cite{tuy}, and globally generated on $\overline{\mathcal{M}}_{0,n}$ in \cite{fakhr}.

\smallskip
In this work we investigate global generation in a more general context. Our main result is:

\begin{thm}\label{thm:GG}
Sheaves of coinvariants defined by simple admissible modules over a vertex operator algebra, strongly generated in degree 1, are globally generated on $\overline{\mathcal{J}}_{0,n}$, and on $\overline{\mathcal{M}}_{0,n}$, if defined.
\end{thm} 

Here we assume that all VOAs are of CFT-type. By \cite{BongLian} the VOAs in Theorem \ref{thm:GG} are quotients of the affinization of a not necessarily reductive Lie algebra structure on their degree $1$ component (see Remark \ref{rmk:BLian}). By \cite{DongMasonIntegrability}, if $V$ is strongly rational (rational, $C_2$-cofinite, simple, and self-contragredient), then $V\cong \bigotimes_{i=1}^rL_{\ell_i}({\mathfrak{g}_i})$, with $\mathfrak{g}_i$ simple Lie algebras,  $\ell_i \in \mathbb{Z}_{>0}$, and $V_1\cong \bigoplus_{i=1}^r\mathfrak{g}_i$.
In Theorem \ref{thm:GG}, $V$ need not be simple, $C_2$-cofinite, or rational, and may for instance be applied to  $L_{\ell}(\mathfrak{g})$, for $\mathfrak{g}$ simple and such that $\ell$ is admissible but not in $\mathbb{Z}_{>0}$. Such VOAs $L_{\ell}(\mathfrak{g})$ are not $C_2$-cofinite, but are quasi-lisse, a natural generalization of $C_2$-cofiniteness, introduced in \cite{ArakawaKawasetsu}.  It follows from \cite[Main Theorem]{ArakawaAffine}, that simple admissible highest weight modules over $L_{\ell}(\mathfrak{g})$ have rational conformal weights, as do more general $V$-modules (see Remark \ref{rmk:integralweights}). In particular, their associated sheaves of coinvariants are defined on $\overline{\mathcal{M}}_{0,n}$ (see Remark \ref{rmk:Descent}).  As in \cite{BongLian}, there are many other examples to which Theorem 1 applies.

By Corollary \ref{cor:fingen}, sheaves described in Theorem 1 are coherent. This improves \cite{an1} giving coherence  on $\mathcal{M}_{0,n}$ for $C_2$-cofinite and self-contragredient $V$, and \cite{DGK} where coherence was proved on $\overline{\mathcal{M}}_{0,n}$ for $C_2$-cofinite $V$.   By Corollary \ref{cor:vb}, such sheaves are vector bundles on $\mathcal{M}_{0,n}$. If $V$ is  $C_2$-cofinite and rational, by \cite{dgt2} these are vector bundles on  $\overline{\mathcal{M}}_{g,n}$, and by Corollary \ref{cor:vbBar}, these vector bundles are globally generated on  $\overline{\mathcal{M}}_{0,n}$, extending \cite{tuy, fakhr}.

While we haven't found conditions to guarantee global generation for $g > 0$, or for VOAs which are not strongly generated in degree 1, to illustrate the subtlety of this problem, we give several representative examples, including: \begin{itemize}
 \item globally generated and positive bundles of coinvariants  (see \S \ref{Holomorphic}, \S \ref{UnitaryVir},  and \S \ref{LatticeVOAs}); and
  \item sheaves of coinvariants that are not globally generated (see \S \ref{UnitaryVir}, and \S \ref{LatticeVOAs});
\end{itemize} 
Influenced by these,  we ask questions and pose potential extensions of Theorem \ref{thm:GG} (see \S \ref{Questions}). 

\smallskip
We next describe our methods, and our findings in more detail.  
\smallskip

Given $n$ simple admissible modules $W^i$ over a vertex operator algebra $V$ of CFT-type, we define a constant sheaf of finite rank, and a morphism of sheaves \eqref{eq:CSheaf} from it to the sheaf of coinvariants (Lemma \ref{CLemma}). Each simple admissible  $V$-module $W^i$  is a direct sum of vector spaces $W^i_d$, graded by the natural numbers, and the fibers of the constant sheaf are isomorphic to the tensor product of the lowest weight spaces $\bigotimes_{i}W^i_0$.  

To prove the sheaf of coinvariants is globally generated, we show  the map  \eqref{eq:CSheaf}  from Lemma \ref{CLemma} is surjective.  For this,  it suffices to prove the induced map on fibers is so, and to achieve this, we use a filtration induced by the grading on $W^\bullet$ by degree.  In particular, we show that all positive degree elements come  from the degree zero part of the filtration, which is naturally a quotient of the constant sheaf.  Crucial to our argument is  Zhu's result that  any simple admissible $V$-module is generated  in degree zero over $V$ \cite[Theorem 2.2.2]{ZhuModular}. 

The proof of surjectivity of \eqref{eq:CSheaf} restricted to fibers at smooth pointed curves is a reinterpretation of the core of the argument of  Tsuchiya, Ueno, and Yamada \cite[Proposition   2.3.1]{tuy},  that implies global generation of bundles defined by affine Lie algebras at integrable levels $\ell \in \mathbb{Z}_{>0}$  on $\mathcal{M}_{0,n}$, as their analysis has essential features in common with ours.  However, at pointed curves with singularities,  due to differences in the definitions of the Lie algebras used to define the coinvariants in these two settings, our proof of surjectivity is considerably more involved than Fakhruddin's proof  \cite{fakhr} of global generation of $\mathbb{V}_0(L_{\ell}({\mathfrak{g}}),W^{\bullet})$ on $\overline{\mathcal{M}}_{0,n}$ for  $\ell \in \mathbb{Z}_{>0}$.

The terms we refer to are defined in \S \ref{sec:Background}, and in \S \ref{sec:coinv} we construct the constant bundle, proving Lemma \ref{CLemma}. Theorem \ref{thm:GG} is proved in \S \ref{ProofThm}, after preparation is given in \S \ref{sec:MainPrep}. Proofs of the Corollaries  are given in \S \ref{sec:corollaries}.  We remark that Corollary \ref{cor:fingen} follows primarily from the proof of Theorem \ref{thm:GG}, while Corollary \ref{cor:vb} is obtained by combining Corollary \ref{cor:fingen} with results from \cite{dgt1}.  Corollary \ref{cor:vbBar} follows from Theorem \ref{thm:GG} and results from \cite{dgt2}.

 If $\mathbb{V}_g(V, W^{\bullet})$ is a globally generated bundle on $\overline{\mathcal{M}}_{g,n}$ for $g>0$, as we next explain, it is necessary that $V$ has non-negative central charge.  Globally generated sheaves will have positive rank, and first Chern classes will be nef  (a divisor is nef if it non-negatively intersects all curves).  Suppose $V$ is $C_2$-cofinite, and rational so  $\mathbb{V}_g(V, W^{\bullet})$ is known to be a vector bundle \cite{dgt2}. If also self-contragredient, then by  \cite[Corollary  2]{dgt3},  first Chern classes of  $\mathbb{V}_g(V, W^{\bullet})$ are linear combinations of tautological classes, including $\lambda$ (the first Chern class of the Hodge bundle).  By \cite[Theorem 2.1]{GKM}, for $g\ge 2$, the coefficient of $\lambda$ is a rational expression involving the rank of the bundle and the central charge of $V$. In particular, the central charge must be nonnegative for $\mathbb{V}_g(V, W^{\bullet})$ to be nef. While not neccessary, strong unitarity of $V$ is sufficient to guarantee that the central charge is positive, and such examples are therefore of interest.

In \S \ref{Questions} we discuss  questions inspired by a number of examples we have studied, represented in simple cases here.  The leitmotif  is that sheaves of coinvariants derived from vertex operator algebras related to affine VOAs seem to be geometric in nature.  
For instance, in Question \ref{Q5} we ask whether $\mathbb{V}_g(\bigotimes_{j=1}^r L_{\ell_j}({\mathfrak{g}_j}), \bigotimes_{j=1}^r W^{\bullet}_{j})$ is isomorphic to $\bigotimes_{j=1}^r\mathbb{V}_g(L_{\ell_j}({\mathfrak{g}_j}), W^{\bullet}_{j})$. In this case, dual sheaves of coinvariants could be identified with generalized theta functions, and would be subject to strange dualities (see Remark \ref{SD}). 
 
Our remaining questions are about positivity. As illustrated in examples given in \S \ref{UnitaryVir},  in cases where the rank of the constant bundle $\mathcal{W}_0^\bullet$ was at least as large as the rank of the sheaf of coinvariants on $\overline{\mathcal{M}}_{0,n}$,   the latter had positivity properties, and otherwise were positive if and only if an integral degree condition was satisfied (see Definition \ref{IC}).  In particular, since Chern classes of globally generated bundles are base point free, we were able to establish non global-generation by checking that if $n=4$, the degree of the sheaf was negative, or for $n\ge 5$ that its first Chern class was not nef.

In \S \ref{Holomorphic} we consider a class line bundles defined by holomorphic VOAs and the affine VOA given by $E_8$. The associated sheaves $\mathcal{W}^\bullet_0$  also have rank one.   While Theorem \ref{thm:GG} does not apply in positive genus, we can see these bundles are globally generated on $\overline{\mathcal{M}}_{g,n}$ for $g>0$.
 
Tools from VOA theory like factorization, Zhu's algebra, and Zhu's character formula, can often help one compute Chern classes, and the ranks of the sheaves of coinvariants and of $\mathcal{W}^{\bullet}_0$, the latter of which is determined by the dimensions of the degree zero components of the module. 
 
In \cite{ZhuModular}, Zhu introduced an associative algebra $A(V)$, and established functors between categories of  $A(V)$-modules and $V$-modules.  By \cite[Theorem 2.2.2]{ZhuModular}, any simple admissible $V$-module $W=\bigoplus_{d\in \mathbb{N}}W_d$-corresponds to the simple $A(V)$ module $W_0$. In particular, if $A(V)$ is commutative, then $W_0$ is one dimensional, and the constant sheaf has rank one. Minimal series principal $W$-algebras $\mathcal{W}^{\ell}(\mathfrak{g})$ and their simple quotients $\mathcal{W}_{\ell}(\mathfrak{g})$ have commutative Zhu algebras (see Example \ref{eg:UnitaryW}).  These rational, $C_2$-cofinite, self-contragredient VOAs have appeared prominently in the literature and are related to other important families of VOAs.  For instance, given $\ell \in \mathbb{Z}_{\ge 2}$, and $r=-\ell+\frac{\ell+1}{\ell+2}$, there is a well-known isomorphism of  $\mathcal{W}_r(\mathfrak{sl}_2, \ell)$ and the Parafermion algebras ${\rm{K}}(\mathfrak{sl}_2, \ell)$ first proven in  \cite{ArakawaLamYamada:2019:parafermion}. A complete list of such isomorphisms between $\mathcal{W}_r(\mathfrak{sl}_n, \ell)$ and ${\rm{K}}(\mathfrak{sl}_2, \ell)$, for any $r$ and $n$, is given in \cite[Theorem 10.3]{Linshaw:2021:universal}. Both $\mathcal{W}^{\ell}(\mathfrak{g})$ and $\mathcal{W}_{\ell}(\mathfrak{g})$ can be realized as cosets of tensor products of affine vertex operator algebras (proved for simply laced $\mathfrak{g}$ in  \cite{ArakawaLamYamada:2019:parafermion, ACL}, and for non-simply laced $\mathfrak{g}$  in \cite{creutzig2021trialities}).   In \S \ref{UnitaryVir}, we consider bundles of coinvariants defined by modules over ${\rm{K}}(\mathfrak{sl}_2, 2)$ equal to the Virasoro VOA $\textit{Vir}_{\frac{1}{2}}$. 

Many VOAs do not have a commutative Zhu algebra.  For instance, generally speaking, affine Lie algebras and even lattice VOAs do not.  One can often compute the rank of the constant bundle $\mathcal{W}_0^\bullet$  using Zhu's chararacter formula, which involves modular forms, as we explain  in \S \ref{LatticeVOAs}.  There we consider sheaves on $\overline{\mathcal{M}}_{0,n}$ generated by modules  over even lattice vertex algebras, which have rank one by \cite{dgt3}. Except in special cases, even lattice VOAs do not coincide with, nor are they constructed from, affine VOAs.  We see that depending on how the modules are chosen, first Chern classes may be negative (so sheaves are not globally generated), zero (so sheaves are constant), and positive (so sheaves are possibly globally generated).

\subsection*{Acknowledgements}We thank Thomas Creuzig, Bin Gui, Yi-Zhi Huang, Jim Lepowsky, Nicola Tarasca, and Simon Wood for answering our questions and for useful discussions. We thank Prakash Belkale, Ron Donagi, and Najmuddin Fakhruddin for comments on a draft of this manuscript. We are truly grateful to a very knowledgeable anonymous referee who informed us of and clarified a number of developments in the VOA literature.

\section{Background}\label{sec:Background} 

Following  \cite{flm, tuy, bzf, dgt2} we briefly give notation and results used here.
We also recommend \cite{dgt3, DGK}, which were written primarily for algebraic geometers.

\subsection{Virasoro (Lie) algebra}\label{Vir}

The {Virasoro (Lie) algebra} $\Vir$ is a one dimensional central extension of the Witt (Lie) algebra $\Der\, \mathcal{K}$.   The Witt algebra represents the functor which assigns to a $\mathbb{C}$-algebra $R$ the Lie algebra $\Der\,\mathcal{K}(R):=R(\!( z)\!) \partial_z$  generated over $R$ by the derivations $L_p:=-z^{p+1}\partial_z$, for $p\in \mathbb{Z}$, with  relations $[L_p, L_q] = (p-q) L_{p+q}$. In particular, we can view the Lie algebra $\Vir$ as representing the functor which assigns to $R$ the Lie algebra generated over $R$ by a formal vector $K$ and the elements $L_p$, for $p\in \mathbb{Z}$, with Lie bracket given by $[K, L_p]=0$ and $[L_p, L_q] = (p-q) L_{p+q} + \frac{K}{12} (p^3-p)\delta_{p+q,0}$.

\subsection{Vertex operator algebras}\label{sec:VOAs}\label{sec:defvoa}

By a {\textit{vertex operator algebra}}, we mean a four-tuple $(V, {\textbf{1}}, \omega, Y(\cdot,z))$, where:
\begin{enumerate}
\item $V=\bigoplus_{i\in \NN} V_i$ is a non-negatively graded  $\mathbb{C}$--vector space with $\dim V_i<\infty$;
\item ${\textbf{1}}$ is an element in $V_0$, called the \textit{vacuum vector};
\item $\omega$ is an element in $V_2$, called the \textit{conformal vector};
\item $Y(\cdot,z)\colon V \rightarrow  \textrm{End}(V)\llbracket z,z^{-1} \rrbracket$, a linear map taking $A \in V$, to $Y(A,z) :=\sum_{i\in\mathbb{Z}} A_{(i)}z^{-i-1}$. The series $Y(A,z)$ is called the \textit{vertex operator} assigned to $A$.
\end{enumerate}

The datum $(V, {\textbf{1}}, \omega, Y(\cdot,z))$ is required to satisfy four axioms which we state in \cite{dgt2}, including axioms that vertex operators satisfy a weak version of commutativity and a weak version of associativity. One may regard this as framing a VOA as generalizing a commutative associative algebra. We highlight here the main properties that we will use in this paper and we refer to \cite{dgt1,dgt2,flm} for more details.
\begin{enumerate}[label=(\roman*)]
\item {\textit{Conformal structure:}} the Virasoro algebra acts on $V$ through the identification $L_{p}=\omega_{(p+1)}$ and $K= c_V \text{Id}_V$ for some complex number $c_V$ called central charge of $V$.
\item {\textit{Vacuum axiom:}} $Y(\mathbf{1},z)= \text{Id}_V$.
\item {\textit{Graded action:}} if $A \in V_k$, then $A_{(j)}V_\ell \subseteq V_{\ell + k -j-1}$.
\item {\textit{Commutator formula:}} \label{it:commV} $[A_{(i)},B_{(j)}] =\sum_{k\ge 0}{\binom{i}{k}} \left(A_{(k)}(B)\right)_{(i+j-k)}$.
\item  {\textit{Associator formula:}} \label{it:assocV} $\left(A_{(i)}(B)\right)_{(j)} =\sum_{k\ge 0}(-1)^k {\binom{i}{k}}\left(A_{(i-k)}B_{(j+k)}-(-1)^iB_{(i+j-k)}A_{(k)}\right)$.
\end{enumerate}

\noindent
A VOA $V$ is {\em{strongly rational}}, or {\em{of CohFT-type}},  if $V$ is simple, self-contragredient, and if $V$ is:\begin{enumerate}[label=(\Roman*)]
\item \label{CohFTI}{\textit{of CFT-type:}} $V=\bigoplus_{i\in \NN}V_i$ with $V_0\cong \mathbb{C} \mathbf{1}$; 
\item \label{CohFTII} {\textit{rational:}} There are finitely many simple $V$-modules, and every finitely generated module is a direct sum of simple modules; 
\item \label{CohFTIII}$C_2$-{\textit{cofinite:}} The space $C_2(V):=\mathrm{span}_{\mathbb{C}}\left\{A_{(-2)}B \,:\, A, B \in V\right\}$ has finite codimension in $V$. 
\end{enumerate}

\subsection{\texorpdfstring{$V$}{V}-Modules}\label{sec:VMOD}
By a $V$-module $W$, we mean what in the literature is known as an \textit{admissible $V$-module}, that is a pair $(W, Y^W(-,z))$ consisting of  \begin{enumerate}
    \item an $\NN$-graded vector space $W=\bigoplus_{i \geq 0} W_i$ with $\dim(W_i)<\infty$ and $W_0 \neq 0$;
    \item a linear map $Y^W(-, z) \colon V \to \End(W)\llbracket z,z^{-1}\rrbracket$,  $A \mapsto Y^W(A,z) = \sum_{i \in \ZZ} A_{(i)}^Wz^{-i-1}$. 
\end{enumerate} In order for this pair to define an admissible $V$-module, certain axioms need to hold \cite{dgt2,fhl,dl}. Instead of reporting all the properties that $(W, Y^W(z,-))$ must satisfy, we list here only those that will be used in this paper (see \cite{dgt2} for more details).  These are: 
\begin{enumerate}[label=(\roman*)]
\item {\textit{conformal structure:}}  the Virasoro algebra acts on $W$ through the identification $L_p \cong \omega^W_{(p+1)}$.
\item {\textit{vacuum axiom:}}  $Y^W(\mathbf{1},z)= \text{Id}_W$.
\item {\textit{graded action:}}  if $A \in V_k$, then $A_{(j)}^WW_\ell \subseteq W_{\ell + k -j-1}$ and we write $\deg(A_{(j)}^W)= \deg(A)-j-1$.
\item {\textit{commutator formula:}} \label{it:commVmod} $[A^W_{(i)},B^W_{(j)}] =\sum_{k\ge 0}{\binom{i}{k}}\left(A_{(k)}(B)\right)^W_{(i+j-k)}$.
\item  {\textit{associator formula:}} \label{it:assocVmod} $\left(A_{(i)}(B)\right)^W_{(j)} =\sum_{k\ge 0}(-1)^k {\binom{i}{k}}\left(A^W_{(i-k)}B^W_{(j+k)}-(-1)^iB^W_{(i+j-k)}A^W_{(k)}\right)$.
\end{enumerate}
In what follows the endomorphism $A^W_{(j)}$ will simply be denoted by $A_{(j)}$. It is important to observe that $V$ is a $V$-module and that the commutator and associator formulas for $V$ and for $V$-modules both arise from the \textit{Jacobi identities} for $V$ and for $V$-modules. Moreover, when $W$ is a simple $V$-module, then there exists an element $\alpha \in \CC$ called the {\it{conformal dimension of $W$}}, such that $L_0(w)=(\alpha + \deg(w)) w$ for every homogeneous element $w \in W$. 

\begin{definition}\label{IC} An $n$-tuple $(W^1,\ldots,W^n)$ of admissible $V$-modules $W^i$ of conformal dimension $\alpha_i$ is said to satisfy the integrality condition, or integrality property, if the sum of conformal dimensions $\sum_{i=1}^n\alpha_i$ is an integer (which can be zero).
\end{definition}

\subsection{Strong finite generation}\label{sec:SGD1} 
A vertex algebra $V$ is called \textit{finitely strongly generated} if there exists finitely many elements $A^1$, $\ldots$, $A^r \in V$  such that $V$ is spanned by the elements of the form 
\[A^{i_1}_{(-n_1)} \cdots A^{i_r}_{(-n_r)} \bf{1},\]
with $r\ge 0$ and $n_i \ge 1$ (see \cite{ArakawaStrong}). We say that $V=\bigoplus_{i\in \mathbb{N}}V_i$ is \textit{strongly generated in degree $d$} if it is possible to choose the generators $A^{i_j}$ to be in $V_m$ for $m \le d$.

By \cite[Proposition 2.5]{Liu}, one has that $V$ is finitely strongly generated if and only if $V$ is $C_1$-cofinite.  If $V$ is $C_2$-cofinite, then it is $C_1$-cofinite. However, the quasi-lisse, but not $C_2$-cofinite affine VOAs defined by simple Lie algebras and admissible, non-integral levels are strongly finitely generated in degree 1.

\begin{remark}
\label{rmk:BLian}
VOAs of CFT-type,
strongly generated in degree 1 were classified in \cite{BongLian}.  More generally, for 
any so-called preVOA $V$ of CFT-type by \cite[Theorem 3.7]{BongLian}, the degree 1 component $V_1$ has the structure of a Lie algebra, with bracket $[A,B]=A_{(0)}(B)$. This Lie algebra $(V_1, [\ , \ ])$, which need not be simple, or reductive, is equipped with a symmetric invariant
bilinear form $<A \ , \ B>=A_{(1)}(B)$. Roughly speaking, in the terminology of \cite{BongLian}, a preVOA satisfies many of the properties of a VOA except
those involving a conformal vector.  Given any pair consisting of a Lie algebra $\mathfrak{g}$ and symmetric invariant bilinear form $< \ , \ >$,  Lian defines the affinization, and proves in
\cite[Theorem 4.11]{BongLian} that for any preVOA $V$ of CFT-type, if strongly generated in degree $1$, then $V$ is isomorphic to a quotient of 
the affinization of $(V_1,< \ , \ >)$ by some ideal.  The last step in the classification is to determine 
which preVOAs admit a Virasoro vector, and have the structure of a VOA. He classifies such Virasoro vectors
(see \cite[Corollary 6.15]{BongLian}).  As pointed out to us by a referee of our paper (who told us of this work), the most interesting aspect of \cite{BongLian} is that this
class of examples is much richer than the affinizations of reductive Lie algebras and their 
quotients. New examples are given in \cite[\S 6.4]{BongLian}.
\end{remark}

\subsection{Coordinatized curves}\label{LieDefs}
As the sheaf of coinvariants on $\overline{\mathcal{M}}_{g,n}$, the constant sheaf constructed in Lemma \ref{CLemma} is defined first on a covering $\tMgn[n]$, of $\overline{\mathcal{M}}_{g,n}$, and then descended to $\overline{\mathcal{M}}_{g,n}$ along two maps which we  recall here.  At the second step, one
 applies Tsuchimoto's method, as in \cite{dgt1} and \cite{dgt2}.  By $\tMgn[n]$ we mean the moduli space of triples $(C, P_{\bullet}, t_{\bullet})$, where $(C, P_{\bullet}) \in \overline{\mathcal{M}}_{g,n}$, and  $t_{\bullet}$ is an $n$-tuple of formal coordinates $t_i$ at each of the marked points $P_i$. This space is described in detail in \cite[\S 2.2.2]{dgt2}. To understand the maps along which the two sheafs are to descend, we next describe the group scheme $\mathrm{Aut}\,\mathcal{O}$, and the varieties ${\mathscr{A}}\textit{ut}_{C/S}$ on which it acts.

Consider the functor which assigns to a $\mathbb{C}$-algebra $R$ the group:
\[
\mathrm{Aut}\,\mathcal{O}(R) = \left\{z \mapsto \rho(z)= a_1 z + a_2 z^2 + \cdots \, | \, a_i \in R, \, a_1 \textrm{ a unit}  \right\}
\]
of continuous automorphisms of the algebra $R\llbracket z \rrbracket$ preserving the ideal $zR\llbracket z \rrbracket$. The  group law  is given by composition of series: $\rho_1 \cdot \rho_2 := \rho_2\circ \rho_1$. This functor is represented by a group scheme, denoted $\text{Aut}\,\mathcal{O}$.  

First suppose $C$ is a smooth curve and let $\mathscr{A}\textit{ut}_C$ be the smooth variety whose set of points are pairs $(P,t)$, with $P \in C$, $t \in \widehat{\mathscr{O}}_P$ and $t\in \mathfrak{m}_P\setminus \mathfrak{m}^2_P$, a formal coordinate at $P$. Here $\mathfrak{m}_P$ is the maximal ideal of $\widehat{\mathscr{O}}_P$, the completed local ring at the point $P$. There is a simply transitive right action of $\mathrm{Aut}\,\mathcal{O}$ on  $\mathscr{A}\textit{ut}_C\rightarrow C$,  given by changing coordinates: 
\[
\mathscr{A}\textit{ut}_C \times \mathrm{Aut}\,\mathcal{O} \rightarrow \mathscr{A}\textit{ut}_C, \qquad \left((P,t),\rho\right) \mapsto \left(P, t\cdot \rho := \rho(t)\right),
\]
making $\mathscr{A}\textit{ut}_C$  a  principal $(\mathrm{Aut}\,\mathcal{O})$-bundle on~$C$. A choice of  formal coordinate  at $P$ gives a trivialization
\[
\mathrm{Aut}\,\mathcal{O}\xrightarrow{\simeq_t}\mathscr{A}\textit{ut}_P, \qquad \rho\mapsto \rho(t).
\]
If $C$ is a \textit{nodal} curve, then to define a principal $(\mathrm{Aut}\,\mathcal{O})$-bundle on $C$ one may give a principal $(\mathrm{Aut}\,\mathcal{O})$-bundle on its normalization, together with a gluing isomorphism between the fibers over the preimages of each node. For simplicity, suppose $C$ has a single node $Q$, and let $\widetilde{C}\rightarrow C$ denote its normalization, with  $Q_+$ and $Q_{-}$ the two preimages of $Q$ in $\widetilde{C}$. A choice of formal coordinates $s_{\pm}$ at $Q_{\pm}$, respectively, determines a smoothing of the nodal curve $C$ over $\mathrm{Spec}(\mathbb{C}\llbracket q\rrbracket)$ such that,  locally around the point $Q$ in $C$, the family is defined by $s_+ s_- = q$.  One may identify the fibers at $Q_{\pm}$  by the  gluing isomorphism induced from the identification $s_+=\gamma(s_-)$:
\begin{equation*}
\label{eq:glisomAut}
\mathscr{A}ut_{Q_{+}} \simeq_{s_+} \mathrm{Aut}\,\mathcal{O}  \xrightarrow{\cong} \mathrm{Aut}\,\mathcal{O} \simeq_{s_-} \mathscr{A}ut_{Q_{-}},
\qquad  \rho(s_+)\mapsto \rho \circ \gamma (s_-),
\end{equation*}
where $\gamma\in \mathrm{Aut}\,\mathcal{O}$ is the involution defined as
\begin{equation*}
\label{eq:muz}
\gamma(z) := \frac{1}{1+z}-1 = -z+z^2-z^3+\cdots.
\end{equation*} This may be carried out in families to define $\mathscr{A}ut_{C/S}\rightarrow C/S$.  The identification of the universal curve $\overline{\mathcal{C}}_g \cong \overline{\mathcal{M}}_{g,1} \rightarrow \overline{\mathcal{M}}_{g}$ leads to the  principal $(\mathrm{Aut}\,\mathcal{O})$-bundle $\widetriangle{\mathcal{M}}_{g,1}\rightarrow \overline{\mathcal{M}}_{g,1}$.

The group scheme $\mathrm{Aut}_+\mathcal{O}$ represents the functor assigning to a $\mathbb{C}$-algebra $R$ the group:
\[
\mathrm{Aut}_+\mathcal{O}(R) = \left\{z \mapsto \rho(z)= z + a_2 z^2 + \cdots \, | \, a_i \in R \right\},
\]
and one has $\mathrm{Aut}\,\mathcal{O} =\mathbb{G}_m \ltimes \mathrm{Aut}_+\mathcal{O}$. In particular,  by \cite{dgt1}, the projection $\widetriangle{\mathcal{M}}_{g,n} \rightarrow \overline{\mathcal{M}}_{g,n}$ is an the $(\mathrm{Aut}\,\mathcal{O})^{\oplus n}$-torsor, and factors as the composition of an $(\mathrm{Aut}_+\mathcal{O})^{\oplus n}$-torsor and a $\mathbb{G}_m^{\oplus n}$-torsor:
\begin{equation*}
\label{eq:torsors}
\xymatrix{
\widetriangle{\mathcal{M}}_{g,n} \ar[rrd]_-{(\mathrm{Aut}_+\mathcal{O})^{\oplus n}\quad} \ar[rrrr]^-{(\mathrm{Aut}\,\mathcal{O})^{\oplus n}} &&&& \overline{\mathcal{M}}_{g,n}\\
&&\overline{\mathcal{J}}_{g,n} \ar[urr]_-{\mathbb{G}_m^{\oplus n}}}
\end{equation*}
\noindent
where $\overline{\mathcal{J}}_{g,n}$ parametrizes objects of type $(C,P_\bullet, \tau_\bullet)$, where $(C,P_\bullet)$ is a stable $n$-pointed genus $g$ curve, and $\tau_\bullet=(\tau_1,\dots,\tau_n)$ with $\tau_i$ a non-zero $1$-jet of a formal coordinate at $P_i$, for each $i$.

\subsection{Lie algebras that act}\label{sec:Coinvariants} To define the sheaf of coinvariants $\mathbb{V}(V;W^\bullet)$ on $\tMgn[n]$, one uses the sheaf of ancillary Lie algebras $\mathfrak{L}(V)^n$, and the sheaf of chiral Lie algebras $\mathcal{L}_{\widetriangle{C}_{g,n} \setminus P_\bullet}(V)$, each of which acts on the tensor product $W^{\bullet}=\bigotimes_{i}W^i$.   The fiber of $\mathfrak{L}(V)^n$ over $(C,P_\bullet, t_\bullet)$ is  given by the direct sum $\bigoplus_{i=1}^n\mathfrak{L}_{P_i}(V)$ of the ancillary Lie algebras
\[\mathfrak{L}_{P_i}(V) :=\dfrac{V \otimes \CC(\!(t_i)\!)}{\text{Im} \nabla} \qquad \text{ with } \nabla = L_{-1} \oplus \partial_{t_i}
\]
with Lie bracket given by
$[A_{[j]}, B_{[k]}]=\sum_{\ell\ge 0}{\binom{j}{\ell}}(A_{(\ell)}(B))_{[j+k-\ell]}$,
where $A_{[{j}]}$ is the class of the element $A\otimes t_i^{j}$ in $\mathfrak{L}_{P_i}(V)$. In particular, the diagonal action of $\bigoplus_{i=1}^n\mathfrak{L}_{P_i}(V)$ on $W^\bullet \otimes \Oo_{\tMgn[n]}$ is induced by the map $\mathfrak{L}_{P_i}(V) \to \End(W^i)$ that takes $A_{[j]}$ to the endomorphism $A^{W^i}_{(j)}$. The degree of every element of $\mathfrak{L}_{P_i}(V)$ is identified with the degree of the associated endomorphism, that is for homogeneous elements $A \in V$ and $j \in \ZZ$ we have $\deg(A_{[j]})= \deg(A)-j-1$.

By \cite{dgt2} there is a coordinate independent version of $\mathfrak{L}_{P_i}(V)$ which we briefly recount. One can define a coordinate independent version of the ancillary Lie algebras as well as give a description of the sheaf of chiral Lie algebras, and their actions on $W^{\bullet}$, using the sheaf of vertex algebras $\mathcal{V}_C$ on the  curve $C$.   The sheaf $\mathcal{V}_C$ was originally defined on smooth curves in \cite{bzf}. To define the sheaf $\mathcal{V}_C$ on a nodal curve $C$, it is enough to define it on open subsets which do not include nodes, and then for each node $Q$, define the sheaf on the normalization of the curve, together with specifying isomorphisms of fibers over the preimages of each node.  This is explained in detail in \cite[\S 2.5]{dgt2}, where it is shown that $\mathcal{V}_C$ is a sheaf of $\mathcal{O}_C$-modules. 
 
If $t_i$ is a local coordinate at $P_i$, then the ancillary Lie algebra of $V$ at $P_i$ is isomorphic to 
\[\text{H}^0\left(D^\times_{P_i}, \mathcal{V}_{C}\otimes\omega_{{C}}/\textrm{Im}\,\nabla\right) \xrightarrow{\; \simeq_{t_i} \;}  \mathfrak{L}_{P_i}(V),\]
and the chiral Lie algebra for $(C,P_{\bullet})$ is defined to be
\[
\mathcal{L}_{C\setminus P_\bullet}(V) := \text{H}^0\left(C\setminus P_\bullet, \mathcal{V}_C\otimes \omega_C/\textrm{Im}\nabla \right), 
\]
and one has a map of sheaves of Lie algebras given by restriction:
\begin{equation}
\label{chiraltoLQp}
\xymatrix{\mathcal{L}_{C\setminus P_\bullet}(V) \ar[r]& \bigoplus_{i =1}^{n}\text{H}^0\left(D^\times_{P_i}, \mathcal{V}_{\widetilde{C}}\otimes\omega_{\widetilde{C}}/\textrm{Im}\,\nabla\right) \ar[r]^-{\simeq_{t_i}} & \bigoplus_{i =1}^{n} \mathfrak{L}_{P_i}(V).}
\end{equation}
In what follows, the image of $\sigma$ in $\bigoplus_{i=1}^{n} \mathfrak{L}_{P_i}(V)$ will be denoted $(\sigma_{P_1}, \dots, \sigma_{P_n})$.

\subsection{Coinvariants in two steps}\label{sec:CoinvDefs} We describe how the sheaf of coinvariants, first defined over $\tMgn[n]$, descends to a sheaf over $\bMgn[n]$. Details of this construction can be found in \cite{dgt2}. 

\subsubsection{}\label{desc} The sheaf of Lie algebras $\mathcal{L}_{\widetriangle{C}_{g,n} \setminus P_\bullet}(V)$ on $\tMgn[n]$  acts on the constant  bundle $W^\bullet \otimes \Oo_{\tMgn[n]}$ so that $\mathcal{L}_{\widetriangle{C}_{g,n} \setminus P_\bullet}(V) \cdot (W^\bullet \otimes \Oo_{\tMgn[n]})$ is a sub-module of $W^\bullet \otimes \Oo_{\tMgn[n]}$.  The sheaf of coinvariants  is  the quotient
\begin{equation}\label{TCoinvariants} \widetriangle{\VV}_g(V;W^\bullet) = \dfrac{W^\bullet \otimes \Oo_{\tMgn[n]}}{\mathcal{L}_{\widetriangle{C}_{g,n} \setminus P_\bullet}(V)(W^\bullet \otimes \Oo_{\tMgn[n]})}.
\end{equation}

We recall that the transitive action given by changing coordinates gives  $\tMgn[n]$ the structure of a principal $({\Aut}\,\mathcal{O})^n$-bundle over $\bMgn[n]$, where  $\pi: \tMgn[n] \longrightarrow \overline{\mathcal{M}}_{g,n}$, is the forgetful map (see \S \ref{LieDefs}).  The actions of $(\Aut \Op)^n$ and of  $\mathcal{L}_{\widetriangle{C}_{g,n} \setminus P_\bullet}(V)$ on $W^\bullet \otimes \Oo_{\tMgn[n]}$ are compatible  \cite{dgt2}, and the action of  $(\Aut \Op)^n$ on $W^\bullet \otimes \Oo_{\tMgn[n]}$ preserves the submodule  $\mathcal{L}_{\widetriangle{C}_{g,n} \setminus P_\bullet}(V)(W^\bullet \otimes \Oo_{\tMgn[n]})$,  inducing an action of $(\Aut \Op)^n$ on $\widetriangle{\VV}_g(V;W^\bullet)$. The sheaf of coinvariants on $\bMgn[n]$ is then defined to be 
\begin{equation*}\label{eq:Sheaf} \VV_g(V;W^{\bullet}) := \left(\pi_*\widetriangle{\VV}_g(V;W^\bullet) \right)^{\Aut \Op^n}.
\end{equation*}

\subsubsection{}\label{desc2} More explicitely, the descent of coinvariants is carried out in two steps: 
\[ \xymatrix{\widetriangle{\mathcal{M}}_{g,n} \ar[rr]^-{\mbox{\tiny{step 1}}} && \overline{\mathcal{J}}_{g,n}  \ar[rr]^-{\mbox{\tiny{step 2}}} && \bMgn[n]}.\]

In the first step, the group scheme $\left(\mathrm{Aut}_+\mathcal{O}\right)^n$ acts equivariantly on $W^\bullet \otimes \Oo_{\tMgn[n]}$, and the quotient by this action descends to a vector bundle $\mathbb{V}^{\mathcal{J}}(V; W^\bullet)$ on $\mathcal{J}=\overline{\mathcal{J}}_{g,n} $, with fibers:
\begin{equation*}\label{Jbundle}
{\mathbb{V}}^{\mathcal{J}}(V; W^\bullet)_{(C,P_\bullet, \tau_\bullet)} = \bigotimes_{i=1}^n \, {W}_{P_i, \tau_i}^i \Big/  \mathcal{L}_{C\setminus P_\bullet}(V)\cdot \left( \bigotimes_{i=1}^n \, {W}_{P_i, \tau_i}^i \right).
\end{equation*} Here, ${W}_{P_i, \tau_i}^i$ is the coordinate-independent realization of the $V$-module $W^i$ assigned at $(P_i,\tau_i)$ as defined in \cite{dgt1}.  In the second step  we then  descend ${\mathbb{V}}^{\mathcal{J}}(V; W^\bullet)$ to $\bMgn[n]$.  The action  $\mathbb{G}_m^n\cong (\mathbb{C}^\times)^n$ is induced by the $\ZZ$-gradation of each $W^i$ and it is described as follows: for $(z_1,\dots,z_n) \in (\mathbb{C}^\times)^n$ and  $w^{\bullet}=w^1\otimes \cdots \otimes w^n \in \bigotimes_jW^j$, given by  homogeneous $w^i \in W$, we set
\begin{equation*}
\label{CactM}
(z_1,\dots, z_n) \cdot w^{\bullet}:= {z_1}^{-{\rm{deg}}(w^1)}w^1 \otimes {z_2}^{-{\rm{deg}}(w^2)}w^2 \otimes \cdots  \otimes {z_n}^{-{\rm{deg}}(w^n)}w^n.
\end{equation*}

When descending along the $\mathbb{G}_m^{\oplus n}$-torsor  to $\overline{\mathcal{M}}_{g,n}$, one applies the following method, inspired by Tsuchimoto in \cite{ts}, and used in  \cite{dgt1,dgt2}. This is explained in detail using a root stack in case the conformal dimensions of modules are rational in \cite[\S 8.7]{dgt2}, while a more general procedure without rationality assumption is described in \cite[Remark 8.7.3, (ii)]{dgt2}. 

\begin{remark}\label{rmk:integralweights} There are a number of sufficient conditions for $V$ that ensure that simple $V$-modules will have rational conformal dimension. For  $L_{k}(\mathfrak{g})$, where $\mathfrak{g}$ is a simple Lie algebra and $k$ is an admissible level that is not a positive integer, as mentioned in the introduction by \cite[Main Theorem]{ArakawaAffine}, the conformal weights of modules in category $\mathcal{O}$ are rational, as they 
are from (slightly) larger categories (that allow for dense modules, spectral flow twists and finite length extensions as studied in for instance in  \cite{CreutzigRidoutWood}) where the weights are determined by those in Arakawa's classification 
by formulas that preserve rationality. Rationality has also been shown using various types of modularity of characters in different contexts \cite{AMRationality, ZhuModular,  MiyamotoModular}. For instance, this is proved for $C_2$-cofinite VOAs in \cite[Corollary  5.10]{MiyamotoModular} using that the span over $\mathbb{C}$ of ordinary and pseudo trace functions is $\operatorname{SL}(2,\mathbb{Z})$ invariant. Admissible affine VOAs have the modular invariance property \cite{KacWakimotoModular}. 
\end{remark}

\subsection{Formulas for the rank} \label{sec:ranksformula} An important feature of the sheaves of coinvariants $\VV_g(V;W^\bullet)$ from a vertex algebra $V$ of CohFT-type, is that their rank can be computed by induction on the genus $g$. This is a consequence of the factorization theorem \cite{dgt2} and can be seen via the equality: 
\begin{equation*} \label{eq:factor1}
    {\rm rank}  \mathbb{V}_g(V; W^\bullet) = \sum_{W \in \mathcal{W}}  {\rm rank} \mathbb{V}_{g-1}(V; W^\bullet\otimes W \otimes W'),
\end{equation*} where the sum is over the finite set of simple admissible $V$-modules $\mathcal{W}$. That $\mathcal{W}$ is finite follows from the assumption that $V$ is rational as well as from the assumption that $V$ is $C_2$-cofinite. Moreover, for every non negative integer $i$ smaller than or equal to $g$ and for every partition $I \sqcup I^c$ of $\{1, \dots, n\}$, also the following equality holds \begin{equation}\label{eq:factor2}
    {\rm rank}  \mathbb{V}_g(V; W^\bullet) = \sum_{W \in \mathcal{W}}  {\rm rank} \mathbb{V}_{g-i}(V; W^I\otimes W ) \,  {\rm rank} \mathbb{V}_{i}(V; W^{I^c}\otimes W').
\end{equation}

\subsection{Formulas for first Chern classes in case \texorpdfstring{$V\cong V'$}{V=V'}} \label{sec:Chern} 

One can use first Chern classes to test the positivity of a given vector bundle. First Chern classes of globally generated vector bundles are base point free, and  they define nef divisors.  A divisor $D$ on a projective variety $X$ is nef if it nonnegatively intersects every curve on $X$.  We recall the formula for first Chern class derived in \cite[Corollary 2]{dgt3} of a vector bundle of coinvariants defined by  a rational, $C_2$-cofinite self-contragredient vertex operator algebra  $V$  of CFT-type, with central charge~$c$, and $n$ simple $V$-modules $W^i$ of conformal dimension $a_i$.
\begin{equation}\label{CC}
c_1\left(  \mathbb{V}_g(V; W^\bullet)  \right)= {\rm rank} \,\mathbb{V}_g(V;W^\bullet) \left( \frac{c}{2}\lambda + \sum_{i=1}^n a_i \psi_i \right)
-b_{\rm irr} \delta_{\rm irr} - \sum_{i,I} b_{i:I} \delta_{i:I},
\end{equation}
where
\[
 b_{\rm irr}=\sum_{W\in\mathcal{W}} a_W \, {\rm rank}\,\mathbb{V}_{g-1}(V; W^{\bullet}\otimes W \otimes W');\]
and 
\[b_{i:I} =\sum_{W\in\mathcal{W}} a_W \, {\rm rank}\,\mathbb{V}_{i}(V;W^I\otimes W) {\rm rank}\,\mathbb{V}_{g-i} (V; W^{I^c}\otimes W').  
\]
In the coefficients $b_{\rm irr}$ of the boundary divisor $\delta_{irr}$ and $b_{i:I}$ of $\delta_{i:I}$ in \eqref{CC}, we sum over the finite set of simple admissible $V$-modules $\mathcal{W}$. 

\subsection{Test curves}\label{sec:FCurves} 
$F$-Curves, which span $A_1(\bMgn[n])$,  are defined to be the numerical equivalence classes of the image of prescribed clutching maps from $\overline{\mathcal{M}}_{1,1}$, and    $\overline{\mathcal{M}}_{0,4}$.  A picture of all possible such maps, and formulas for the intersections of $F$-curves with divisors is given in \cite{GKM}.  The $F$-curves on $\overline{\mathcal{M}}_{0,n}$ are given by a partition $\{1,\ldots,n\} = N_1 \cup N_2 \cup N_3 \cup N_4$ into four nonempty sets, which determines a map from $\overline{\mathcal{M}}_{0,4}$ to $\overline{\mathcal{M}}_{0,n}$, where given $(C,q_{\bullet})\in \overline{\mathcal{M}}_{0,4}$  we obtain a point in $\overline{\mathcal{M}}_{0,n}$ by attaching to $q_i$ for $i\in \{1,\ldots, 4\}$, any stable $|N_i| +1$ pointed curve of genus $0$ by gluing $q_i$ to the $+1$ point.  The $F$-curve, denoted $F_{N_1,N_2,N_3,N_4}$,  is defined to be the numerical equivalence class of the image of this map.

\section{Constant sheaves associated to coinvariants} \label{sec:coinv}

Here we show how to associate to any $n$-tuple $(W^{1}, \ldots, W^n)$ of admissible $V$-modules, a sheaf $\mathcal{W}^{\bullet}_k$ based on the degree $k$ part of the standard filtration $\mathcal{F}_k(\bigotimes_jW^{j})$, defined in \S \ref{Filtration}. Each of these sheaves, considered in \S \ref{sec:ConstantBundle0}, descends from a constant sheaf on $\tMgn[n]$, and in  case $k=0$, remains constant.  In Example \ref{eg:UnitaryW}, we discuss the rank one constant sheaves associated to any $n$-tuple of representations over the minimal series principal $W$-algebra related in some cases to the parafermions and the discrete series Virasoro VOAs, considered in \S \ref{UnitaryVir}.

\subsection{Filtration}\label{Filtration}  The standard filtration on the sheaf of coinvariants  $\widetriangle{\VV}_g(V;W^\bullet)$ on $\tMgn[n]$ defined in \eqref{TCoinvariants} is induced from  a filtration on $\mathcal{L}_{C\setminus P_ \bullet}(V)$, given for $k\in \mathbb{N}$ by
\[
\mathcal{F}_k\,\mathcal{L}_{C\setminus P_ \bullet}(V) := \left\{\sigma \in \mathcal{L}_{C\setminus P_ \bullet}(V) \, | \, \deg \sigma_{P_i}\leq k, \, \mbox{for all $i$}  \right\},
\]
so $\mathcal{L}_{C\setminus P_ \bullet}(V)$ is a filtered Lie algebra. There is also a filtration on $W^\bullet$ defined for $k\in \mathbb{N}$ by
\[
\mathcal{F}_{k}W^{\bullet}=\bigoplus_{0\leq d \leq k}W^{\bullet}_d, \quad\mbox{ where }\quad
W^{\bullet}_d:=\sum_{d_1+\cdots+d_n=d} W^1_{d_1} \otimes \cdots \otimes W^n_{d_n}.
 \]
Since $\mathcal{F}_{k}\,\mathcal{L}_{C\setminus P_ \bullet}(V)\cdot \mathcal{F}_{\ell}W^{\bullet} \subset \mathcal{F}_{k+\ell}W^{\bullet}$, it follows that $W^\bullet$ is a filtered $\mathcal{L}_{C\setminus P_ \bullet}(V)$-module, and we set
\begin{equation*}
\mathcal{F}_{k} \left( W^\bullet_{\mathcal{L}_{C\setminus P_ \bullet}(V)}\right) :=
\left( \mathcal{F}_{k} W^\bullet + \mathcal{L}_{C\setminus P_ \bullet}(V)\cdot  W^\bullet\right)\big/\mathcal{L}_{C\setminus P_ \bullet}(V) \cdot W^\bullet.
\end{equation*}

\subsection{Sheaves \texorpdfstring{$\mathcal{W}^{\bullet}_k$}{Wk}}\label{sec:ConstantBundle0}

In this section we consider an $n$-tuple $(W^1,\ldots, W^n)$ of simple admissible $V$-modules.

\begin{lemma}\label{CLemma} Given $n$ admissible $V$-modules $W^j$, there is a natural map 
\begin{equation}\label{eq:CSheafJ} 
\phi^\mathcal{J} \, \colon \, (\mathcal{W}^{\bullet}_k)^\mathcal{J} = \left(\mathcal{F}_k(W^\bullet) \otimes \pi_* \ \Oo_{\tMgn[n]}  \right)^{\Aut_+\Op^n} \longrightarrow \  \left(\pi_* \ \widetriangle{\VV}_g(V;W^\bullet) \right)^{\Aut_+\Op^n} = {\VV}^\mathcal{J}(V;W^\bullet).
\end{equation}
from the sheaf $(\mathcal{W}^{\bullet}_k)^\mathcal{J}$ on $\mathcal{J}=\overline{\mathcal{J}}_{g,n}$ with fibers  at closed points given by
$\left(\mathcal{W}_k\right)_{(C,P_{\bullet})} \cong \mathcal{F}_k(W^\bullet)$. If the sheaf of coinvariants descends to $\bMgn[n]$, then $\phi^\mathcal{J}$ descends to a map
\begin{equation}\label{eq:CSheaf} 
\phi \, \colon \, \mathcal{W}^{\bullet}_k = \left(\mathcal{F}_k(W^\bullet) \otimes \pi_* \ \Oo_{\tMgn[n]}  \right)^{\Aut \Op^n} \longrightarrow \  \left(\pi_* \ \widetriangle{\VV}_g(V;W^\bullet) \right)^{\Aut \Op^n} = {\VV}_g(V;W^\bullet).
\end{equation}
of sheaves over $\bMgn[n]$. In case $k=0$,  the sheaf $\mathcal{W}^{\bullet}_0$ (resp. $(\mathcal{W}^{\bullet}_0)^\mathcal{J}$) on $\overline{\mathcal{M}}_{g,n}$ (resp. on $\mathcal{J}$) is constant, with fibers $\left(\mathcal{W}_0\right)_{(C,P_{\bullet})}=\bigotimes_j W^j_0$.
\end{lemma}

\begin{remark}\label{rmk:Descent} If the conformal dimensions of the modules are rational, then the sheaf of coinvariants ${\VV}^\mathcal{J}(V;W^\bullet)$ on $\mathcal{J}=\overline{\mathcal{J}}_{g,n}$ descends to the sheaf ${\VV}_g(V;W^\bullet)$ on $\bMgn[n]$ (see \cite[\S 8]{dgt2}).
\end{remark}

\proof
Consider the constant bundle $\mathcal{F}_k(\bigotimes_jW^{j}) \otimes \Oo_{\tMgn[n]}$  on $\tMgn[n]$, where $\mathcal{F}_k(\bigotimes_jW^{j})$ is the degree $k$ part of the filtration defined in \S \ref{Filtration}. Then $\mathcal{F}_k(\bigotimes_jW^{j}) \otimes \Oo_{\tMgn[n]}$ is a sub-bundle of $W^\bullet \otimes \Oo_{\tMgn[n]}$ and there is a natural composition 
  \begin{equation}\label{eq:TMap} \xymatrix{
  \mathcal{F}_k(\bigotimes_jW^{j}) \otimes \Oo_{\tMgn[n]} \ar@{^(->}[r] & W^\bullet \otimes \Oo_{\tMgn[n]}
  \ar[r] &  \dfrac{W^\bullet \otimes \Oo_{\tMgn[n]}}{\Ll_{\widetriangle{C}_{g,n} \setminus P_\bullet}(W^\bullet \otimes \Oo_{\tMgn[n]})}=\widetriangle{\VV}_g(V;W^\bullet).}
 \end{equation}
To show that this induces the maps \eqref{eq:CSheafJ} and \eqref{eq:CSheaf} it is enough to show that $\mathcal{F}_k(\bigotimes_jW^{j})$ is an ${\Aut}\,\mathcal{O}^n$-equivariant subset of $W^\bullet$. For this purpose, recall that ${\Aut}\,\mathcal{O}= \Gm \ltimes {\Aut}_+\,\mathcal{O}$. The action of ${\Aut}_+\mathcal{O}^n$ is given exponentiating the action of $L_i$ for $i \geq 1$, which shifts the degree in the negative direction, hence preserving $\mathcal{F}_k(W^\bullet)$. By definition, an element $z \in \Gm$ sends a homogeneous element $w \in W^j$ to $z^{-\deg \, w} w$, hence it preserves the degree of the element. It follows that $\mathcal{W}^{\bullet}_k:=\left(\mathcal{F}_k(W^\bullet) \otimes \pi_* \Oo_{\tMgN}  \right)^{\Aut \Op^n}$ is well defined, and the morphisms \eqref{eq:CSheafJ} and \eqref{eq:CSheaf} is induced from \eqref{eq:TMap}.

For the last claim, it is enough to show that every element of ${\Aut}\,\mathcal{O}^n$ acts on $W_0^{\bullet}=\bigotimes_j W_0^j$ as the identity. From what we have just observed, every element of ${\Aut}^+\mathcal{O}$ acts on $W^j_0$ as the identity because the modules are positively graded. From the description of the action of $\Gm$ on $W^j$ given above, we have that the action of $z_\bullet=(z_1, \dots, z_n) \in \Gm^n$  on the element $w^{\bullet}=w_0^1\otimes \cdots \otimes w_0^n \in W_0^{\bullet}$, is given by
\[z_\bullet\cdot w^\bullet=z_1^{\deg(w^1_0)}w^1_0 \otimes z_2^{\deg(w^2_0)} w^2_0 \otimes \cdots \otimes z_n^{\deg(w^n_0)} w^n_0 =w^\bullet,\] so the restriction of this action to $W_0^{\bullet}$ is  the identity as wanted. It follows that the action of $(\Aut\,\Op)^n$ on $W_0^\bullet \otimes \Oo_{\tMgN}$ is only given by the action of $(\Aut\,\Op)^n$ on $\Oo_{\tMgN}$ and hence 
\[ \pi_*\left( W_0^\bullet \otimes \Oo_{\tMgN} \right)^{(\Aut\,\Op)^n} = W_0^\bullet \otimes \left(\pi_* \Oo_{\tMgN}\right)^{(\Aut\,\Op)^n} = W_0^\bullet \otimes \Oo_{\bMgN},
\]which concludes the proof.
\endproof

\begin{example}\label{eg:UnitaryW} When $\mathfrak{g}$ is simply laced, the minimal series 
principal $\mathcal{W}$-algebras $\mathcal{W}_{\ell}(\mathfrak{g})$ are simple, rational, $C_2$-cofinite, and of CFT-type
for any (nondegenerate) admissible level $\ell$ \cite{ArakawaRatW, ArakawaC2W}. In case $\ell+h^{\vee} = \frac{k+h^{\vee}}{k+h^{\vee}+1}$, for any positive integer $k$, these algebras are unitary.  $\mathcal{W}_{\ell}(\mathfrak{g})$ is the simple quotient of the universal $W$-algebra $\mathcal{W}^{\ell}(\mathfrak{g})$. Zhu's algebra  $A(\mathcal{W}^{\ell}(\mathfrak{g}))$ is isomorphic to the center $Z(U(\mathfrak{g}))$  of the universal enveloping algebra of $\mathfrak{g}$, and  $A(\mathcal{W}_{\ell}(\mathfrak{g}))$ is a quotient of $Z(U(\mathfrak{g}))$ (see eg \cite{DeSole.Kac:2006:Finite}).  In particular, these algebras are commutative. Since the irreducible representations of a commutative algebra are one dimensional,  any  constant sheaf $\mathcal{W}^{\bullet}_0$ made from  simple modules over $\mathcal{W}^{\ell}(\mathfrak{g})$ or $\mathcal{W}_{\ell}(\mathfrak{g})$ on $\overline{\mathcal{M}}_{0,n}$  is a line bundle.  

As mentioned in the introduction, by \cite[Main Theorem 2]{ACL} in types $A$, $D$ and $E$, and in \cite[Corollary 4.1 and Theorem 7.1]{creutzig2021trialities}) in types $B$ and $C$, both $\mathcal{W}^{\ell}(\mathfrak{g})$ and $\mathcal{W}_{\ell}(\mathfrak{g})$ can be realized as cosets of tensor products of affine vertex operator algebras.
The result in type $A$ with $k=1$ was also proved in \cite{ArakawaLamYamada:2019:parafermion}. In particular, except possibly when $k+h^{\vee} \in \mathbb{Q}_{\leq0}$,
\[\mathcal{W}^{\ell}(\mathfrak{g})\cong Com(V_{k+1}({\mathfrak{g}}), V_k(\mathfrak{g})\otimes L_1(\mathfrak{g})), \ \mbox{ and } \
\mathcal{W}_{\ell}(\mathfrak{g})\cong Com(L_{k+1}({\mathfrak{g}}), L_k(\mathfrak{g})\otimes L_1(\mathfrak{g})).\]
\end{example}

\begin{remark}\label{rmk:BinLattice} Although by \cite{dgt3}, bundles of coinvariants for even lattice VOAs have rank one on $\overline{\mathcal{M}}_{0,n}$, it is not true that Zhu's algebra $A(V_{L})$ will be  commutative, even for an even lattice of rank $1$. \end{remark}

\section{Preparations for the proof of Theorem \ref{thm:GG}}\label{sec:MainPrep}
Here we develop some tools that will be useful for proving Theorem \ref{thm:GG}. Although the theorem only discusses vertex algebras which are strongly generated in degree 1, we first analyze properties of vertex algebras strongly generated in degree $d$ for $d \geq 1$, and then restrict to $d=1$.

\begin{lemma}\label{lem:gendegd} Assume that $V$ is strongly finitely generated in degree $d$ and $W$ is an admissible $V$ module. Then every element $w \in W$ can be written as a linear combination of elements of the form
\begin{equation} \label{eq:gendegd} A^1_{-j_1} A^2_{-j_2} \cdots A^m_{-j_m}w_0,
\end{equation} for some $m \geq 0$, with $A^i \in \mathcal{F}_{d}(V) \setminus V_0$, and such that 
\begin{enumerate}
\item if $\deg(w)=0$, then  $\deg(A^m_{-j_m})\geq 0$, and 
\item if $\deg(w)>0$, then  $\deg A^m_{-j_m} \geq 1$.
\end{enumerate}
\end{lemma}

\proof We begin by observing that we can exclude the case that $A^i \in V_0 = \mathbb{C}\mathbf{1}$ since $\mathbf{1}_{-j}$ either acts on $W$ by zero or by the identity.

By the proof of \cite[Theorem 2.1.2]{ZhuModular}, every element in $w\in W$ can be written as a combination of  ${B^1_{-j_1} B^2_{-j_2} \cdots B^\ell_{-j_\ell}u_0}$ for some $B^i$ in $V$ and $u_0 \in W_0$. Since $V$ is strongly generated in degree $d$, we know that every element $B^i$ of $V$ can be written as a combination of elements of the type $B^{i,1}_{-k^i_1} \cdots B^{i,n_i}_{-k^i_{n_i}} \mathbf{1}$ where ${\rm{deg}}(B^{i,s}) \leq d$ for all $s \in \{1,\dots, n_i\}$.

Using this notation, we are then left to prove that every element written as

\[ w=\left(B^{1,1}_{-k^1_1} \cdots B^{1,n_1}_{-k^1_{n_1}} \mathbf{1} \right)_{-j_1}\cdot \left(B^{2,1}_{-k^2_1} \cdots B^{2,n_2}_{-k^2_{n_2}} \mathbf{1} \right)_{-j_2} \cdots
 \left(B^{\ell,1}_{-k^{\ell}_1} \cdots B^{\ell,n_{\ell}}_{-k^{\ell}_{n_{\ell}}} \mathbf{1} \right)_{-j_{\ell}} \cdot u_0 
\] can be rewritten as a linear combination of elements as in equation \eqref{eq:gendegd}. This result is true by repeated use of the associator formula  applied from left to right. That is, we first rewrite
\[w_1:=\left(B^{2,1}_{-k^2_1} \cdots B^{2,n_2}_{-k^2_{n_2}} \mathbf{1} \right)_{-j_2} \cdots
 \left(B^{\ell,1}_{-k^{\ell}_1} \cdots B^{\ell,n_{\ell}}_{-k^{\ell}_{n_{\ell}}} \mathbf{1} \right)_{-j_{\ell}} \cdot u_0 
\] so that 
\[w = \left(B^{1,1}_{-k^1_1}\cdot B^{1,2}_{-k^1_2} \cdots B^{1,n_1}_{-k^1_{n_1}} \mathbf{1} \right)_{-j_1} \hspace{-2mm} w_1\\
= \left( B^{1,1}_{-k^1_1}\underset{=: \, D^{1,2}}{\left(B^{1,2}_{-k^1_2} \cdots B^{1,n_1}_{-k^1_{n_1}} \mathbf{1}\right)} \right)_{-j_1}\hspace{-4mm}w_1= \left(B^{1,1}_{-k^1_1}(D^{1,2}) \right)_{-j_1} \hspace{-2mm}w_1.\]
We can then expand $\left(B^{1,1}_{-k^1_1}(D^{1,2}) \right)_{-j_1}$ using the associator formula.  Following the expansion, we rewrite $D^{1,2}$ as $B^{1,2}_{-k^1_2}(D^{1,3})$, where $D^{1,3}=\left(B^{1,3}_{-k^1_3} \cdots B^{1,n_1}_{-k^1_{n_1}}\right)$, and again expand using the associator formula. Repeating this for  all $D^{1,i}=\left(B^{1,i}_{-k^1_i} \cdots B^{1,n_1}_{-k^1_{n_1}}\right)$, and expanding using the associator formula, and then carrying out the same procedure for the factors of $w_1$, we arrive at a linear combination of terms of the form described in \eqref{eq:gendegd}.

We note that once in the form given in \eqref{eq:gendegd}, if the term $A^m_{-j_m}$ adjacent to $w_0$ has negative degree then, $A^m_{-j_m}w_0=0$, since $W$ is graded by $\mathbb{N}$.  If $A^m_{-j_m}$ has degree zero then $A^m_{-j_m}w_0=u_0 \in W_0$, and we may as well replace it.
\endproof

Given Lemma \ref{lem:gendegd}, we make the following definition for the length of an element in $W$, which will be used to argue by induction in the proof of Lemma \ref{lem:gendegrefined}.

\begin{definition} Suppose $V$ is strongly generated in degree $d$ and that $W$ is a simple $V$-module.  For $\ell \in  \NN$,  set
 \[G^{\ell}(W):={\rm{Span}}\{A^{1}_{-j_1}\cdots A^{\ell}_{-j_{\ell}}w_0 \, \text{ such that } \, A^{j} \in \mathcal{F}_{d}(V)\setminus V_0, \ w_0 \in W_0\}.\] We say that $w$ has \textit{$d$-length} $\ell$ if $w \in L^\ell(W):= G^{\ell}W \setminus G^{\ell -1}(W)$. 
\end{definition}

When there is no ambiguity on $d$, we will simply use length in place of $d$-length. Given two elements $w_1$ and $w_2 \in W$, we say that $w_1$ is shorter than $w_2$ if the length of $w_1$ is smaller than the length of $w_2$. For the proof of Theorem \ref{thm:GG} we will need a refined version of Lemma \ref{lem:gendegd}.

\begin{lemma} \label{lem:gendegrefined} Assume that $V$ is strongly finitely generated in degree $d$. Then every element $w \in W$ such that ${\rm{deg}}(w)>0$ can be written as a combination of elements of the form as in \eqref{eq:gendegd} with the additional properties:
\begin{enumerate}
\item[(i)] If $\deg(A^1)=1$, then $j_1 \geq 1$.
\item[(ii)] In case $d=1$, every element $A^i_{-j_i}$ has positive degree, or equivalently $j_i \geq 1$ for every $i$.
\end{enumerate}  
\end{lemma}

\proof We start proving part \textit{(ii)} and observe that by induction on the length of the elements it is enough to consider only the case in which $m=2$. Using the commutator formula, we have that
\[
A^1_{-j_1} A^2_{-j_2} w_0 = A^2_{-j_2}A^1_{-j_1} w_0 + \sum_{k \geq 0} \left( A^1_kA^2\right)_{-j_1-j_2-k}w_0.
\] We now show that each non zero term on the right hand side of the equality is a term which is either of degree zero or it satisfies the wanted property. The first term of the right hand side is either zero if $\deg(A^1_{-j_1})$ is negative, or $A^2_{-j_2} u_0$ for some $u_0 \in W^0$  if $\deg(A^1_{-j_1})=0$, or the degree of $A^1_{-j_1}$ is positive.  We now look at the other terms. Since both $A^1$ and $A^2$ have degree $1$, we have that the only non zero summands are those where $k=0$ or $k=1$. When $k=0$ we have that $B=A^1_0A^2$ is an element of degree one, hence we reduce the statement to the case $m=1$. When $k=1$ the element $A^1_1A^2$ has degree zero, which implies that is a multiple of the vacuum vector and so it can act on $w_0$ only by a scalar. 

We are left to prove \textit{(i)}. Also in this case the proof follows from the commutator formula and induction on the length of elements. With more details, let $w = A^1_{-j_1} \cdots A^m_{-j_m} u_0$ as in Lemma \ref{lem:gendegd} and define the $K$-value of $w$, denoted $K(w)$, as the smallest integer in $\{0,1, \dots, m\}$ such that $\deg(A^{K(w)}) \geq 2$, with $K(w)=0$ if all the elements have degree $1$. It is enough to show that every element with $K(w) \geq 2$ can be written as a sum of shorter elements and elements with smaller $K$-value. By repeating the argument we reduce to the case in which either $K=0$ or $K=1$.  Observe that if $K=0$, then we are done from \textit{(ii)} above, while if $K=1$, we are in the situation $\deg(A^1)\geq 2$. When $K(w)=K \geq 2$, then using the commutator formula we can write $w$ as
\begin{multline*}A^1_{-j_1} \cdots  A^K_{-j_K}  \cdot A^{K-1}_{-j_{K-1}} \cdot A^{K+1}_{-j_{K+1}} \cdots A^m_{-j_m} u_0 + \\
+ \sum_{k\geq 0} A^1_{-j_1} \cdots A^{K-2}_{-j_{K-2}} \cdot  (A^{K-1}_{k}(A^K))_{-j_{K}-j_{K-1}-k} \cdot A^{K+1}_{-j_{K+1}} \cdots A^m_{-j_m} u_0. 
\end{multline*}
The first term is an element with $K$-value less than $K$. Moreover, since $\deg(A^{K-1})=1$, then for every $k \geq 0$ we have that $\deg A^{K-1}_{k}(A^K) \geq d$, which shows that the terms in the second line are shorter than $w$. \endproof

\section{Proof of Theorem \ref{thm:GG}}\label{ProofThm}
Here we prove Theorem \ref{thm:GG}.  We recall that in \S \ref{desc} we describe the descent of coinvariants from  $\tMgn[n]$ to $\bMgn[n]$,  which is carried out in two steps, first to a vector bundle $\mathbb{V}^{\mathcal{J}}(V; W^\bullet)$ on $\overline{\mathcal{J}}_{g,n} $, and then, if possible, to $\bMgn[n]$.  To show that the sheaf of coinvariants is globally generated, we show that the maps $\phi^{\mathcal{J}}$ and $\phi$ (when it exists) defined in Lemma \ref{CLemma} from the constant bundle $\mathcal{W}^{\bullet}_0$ to  $\mathbb{V}_0(V; W^{\bullet})$ is surjective. For this, it is sufficient to show that the map under consideration is surjective on fibers. 

Since fibers of either sheaf are isomorphic, and the arguments are the same, without loss of generality we make the argument for the map $\phi$ and assume $(C, P_{\bullet}) \in \overline{\mathcal{M}}_{0,n}$. Then the restriction  $\phi|_{(C, P_{\bullet})}$ of $\phi$ to the fiber  is induced by the composition
\[W^{\bullet}_0 \hookrightarrow W^{\bullet} \twoheadrightarrow 
W^{\bullet}/ \mathcal{L}_{C\setminus P_ \bullet}(V) W^{\bullet} \cong  W^\bullet_{\mathcal{L}_{C\setminus P_ \bullet}(V)}.\] 

It is enough to show that every element of $W^\bullet$ can be written as a combination of elements of $W_0^\bullet$ and elements in $\mathcal{L}_{C \setminus P_\bullet}(V) \cdot W^\bullet$. 

For this purpose it suffices to show that for every $d \geq 1$, we can write any element $w^\bullet$ of $W^\bullet_d = \mathcal{F}_{d}(W^\bullet) \setminus \mathcal{F}_{d-1}(W^\bullet)$ as a linear combinations of element in $\mathcal{L}_{C \setminus P_\bullet}(V) \cdot W^\bullet \cup \mathcal{F}_{d-1}(W^\bullet)$. In other words we will prove that there exist elements: 
\[ \sigma \in \mathcal{L}_{C\setminus P_ \bullet}(V) \qquad \text{ and } \qquad v^{\bullet} \in W^{\bullet},\] so that
\begin{equation}   \label{eq:DegreeId} \sigma(v^{\bullet})-w^{\bullet} \in  \mathcal{F}_{\deg(w^\bullet)-1}(W^{\bullet}).
\end{equation}

In view of part (ii) of Lemma \ref{lem:gendegrefined}, we can further assume that $w^\bullet$ is of the form $w^1 \otimes \cdots \otimes w^n$ where: 
\begin{itemize}
\item every $w^k$ is a homogeneous element of $W^k$ of degree $d_k$;
\item for one $i \in \{1, \dots,n \}$ we can write $w^i=A_{-j}u^i$, with $u^i$ a homogeneous elements of $W^i$, $A \in V_1$ and $j \geq 1$.
\end{itemize}

We start by showing that there exists an element $\sigma = A \otimes \mu \in \mathcal{L}_{C\setminus P_ \bullet}(V)$, with a pole of order $j$ at $P_i$, and regular at all the other points. Since the description of $\mathcal{L}_{C\setminus P_ \bullet}(V)$ over nodal curves requires some extra work, the proof continues treating separately the cases of $C$ being a smooth or a nodal curve.

\subsection{The smooth case: \texorpdfstring{$C \cong \PP^1$}{}}\label{sec:CaseOne}
Since $C\setminus P_{\bullet}$ is affine, we can deduce that
\[\mathcal{L}_{C\setminus P_{\bullet}}(V)\cong \text{H}^0(C\setminus P_{\bullet}, \mathcal{V}_C\otimes \omega_C/\text{Im} \nabla)\cong 
\text{H}^0(C\setminus P_{\bullet}, \mathcal{V}_C\otimes \omega_C)/\nabla \text{H}^0(C\setminus P_{\bullet}, \mathcal{V}_C).\]
By \cite{dgt2},  one has
\[\text{H}^0(C\setminus P_{\bullet}, \mathcal{V}_C\otimes \omega_C)\cong \bigoplus_{m \ge 0}\text{H}^0(C\setminus P_{\bullet}, (\omega_C^{1-m})^{\dim V_m}) \cong \bigoplus_{m\ge 0}V_m \otimes \text{H}^0(C\setminus P_{\bullet}, \omega_C^{1-m}).\]
Using Riemann-Roch, if $D=\mathscr{O}_C(jP_i)$, the subset
\[\bigoplus_{m\ge 0}V_m \otimes \text{H}^0(C, \omega_C^{1-m}(D))\subseteq \bigoplus_{m\ge 0}V_m \otimes \text{H}^0(C\setminus P_{\bullet},  \omega_C^{\otimes 1-m}),\]
is non zero. In particular, it contains an element $\sigma=A\otimes \mu$ where $\mu \in \text{H}^0(C, \Oo(jP_i))$ has a pole of order $j$ at $P_i$, and regular elsewhere.

We next show that the element $v^\bullet= w^1 \otimes \cdots \otimes w^{i-1} \otimes u^i \otimes w^{i+1} \otimes \cdots w^n$, satisfies \eqref{eq:DegreeId}. The action of $\mathcal{L}_{C\setminus P_\bullet}(V)$ on $W^\bullet$ is given by the diagonal action of $\sigma_{P_k}$ on $W^k$ (this is the image of $\sigma$  along the map $\mathcal{L}_{C\setminus P_\bullet} \rightarrow \mathfrak{L}_{P_k}(V)$ arising from \eqref{chiraltoLQp}). From the definition of $\sigma$ we have that 
\begin{equation*} \sigma_{P_i}=A_{[-j]}+\sum_{m \ge 1}\alpha_{-j+m}A_{[-j+m]},
\end{equation*}
and that since $\sigma$ is regular at the other points we have that $\sigma_{P_k}=\sum_{q \geq 0}b_qA_{[q]}$ with $b_q \in \mathbb{C}$, hence
\begin{equation}\label{eq:boundsigmaP}\sigma_{P_k} \cdot W^j_{d_k} \subset W^j_{d_k}.
\end{equation}

Summarizing, we conclude that
\begin{align}\label{Halloween} \sigma(v^{\bullet})-w^{\bullet}
&=\sum_{m \ge 1}\alpha_{-j+m}  \left( w^1\otimes \cdots \otimes w^{i-1}\otimes A_{(-j+m)}u^i \otimes w^{i+1} \otimes \cdots \otimes w^n\right)\\
 & \qquad + \sum_{k \ne i}w^1\otimes \cdots \otimes u^i \otimes  
 \cdots \otimes w^{k-1} \otimes \sigma_{P_k}(w^k)  \otimes w^{k+1} \otimes \cdots w^n. \nonumber
\end{align}

Terms in the first line of the right hand side of \eqref{Halloween} are in $\mathcal{F}_{\deg(w^\bullet)-m}(W^{\bullet})$ for $m \ge 1$. Each summand in the second line of \eqref{Halloween} is in $\mathcal{F}_{\deg(w^\bullet)-j}(W^{\bullet})$ by \eqref{eq:boundsigmaP}. Since $j \geq 1$ by assumption, then we can conclude that \eqref{eq:DegreeId} holds, concluding the proof in the smooth case.

\subsection{\texorpdfstring{$C$}{C} is nodal.}\label{sec:Thm1nodalCase} From the stability condition we can ensure that $C\setminus P_\bullet$ is an affine curve. Moreover, without loss of generality we can assume that $C$ has two components, $C_+$ and $C_-$ which meet at only one node $Q$ as in Figure \ref{fig:Cnodal}. The marked points on $C_+$ will be indexed by $P_{\bullet,+}$ and the marked points on $C_-$ will be indexed by $P_{\bullet,-}$. Assume that the point $P_i=P$ lies in the component $C_+$. The preimages of $Q$ via the normalization morphism $\eta \colon \widetilde{C}= C_+ \sqcup C_- \to C$ are the points $Q_+$ and $Q_-$ with local coordinates $t_+$ and $t_-$.

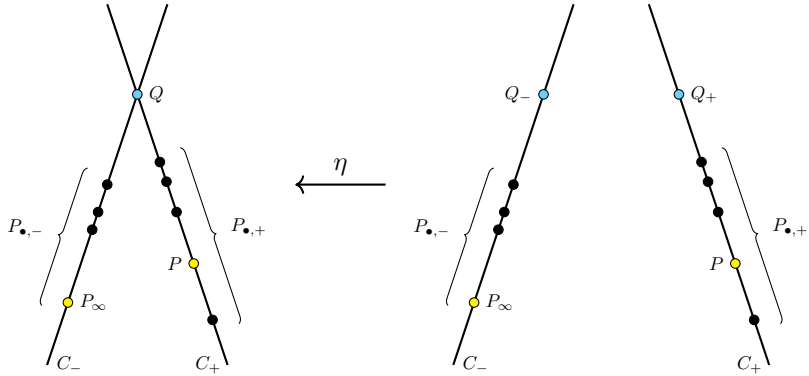
\begin{figure}[h!]
    \centering
\begin{tikzpicture}[scale=0.6]

\draw[thick] (-2,-6) -- (2/3,2);

\filldraw[color=black] (-2/2,-6/2) circle[radius=3pt];
\filldraw[color=black] (-2/2.3,-6/2.3) circle[radius=3pt];
\filldraw[color=black] (-2/3,-6/3) circle[radius=3pt];
\draw[fill=yellow] (-2/1.3,-6/1.3) circle[radius=3pt] node [right,scale=0.7] {$\; P_\infty$};
\node [right,scale=0.7] at (-2,-6) {$\, C_-$};
\draw [decorate, decoration={brace, raise=9pt}] (-2/1.3*1.05,-6/1.3*1.05) -- (-2/3*0.9,-6/3*0.9) node [midway, above left,scale=0.7] {$P_{\bullet, -} \qquad$};

\draw[thick] (2,-6) -- (-2/3,2);
\filldraw[color=black] (2/1.2,-6/1.2) circle[radius=3pt];
\filldraw[color=black] (2/4,-6/4) circle[radius=3pt];
\filldraw[color=black] (2/2.3,-6/2.3) circle[radius=3pt];
\filldraw[color=black] (2/3.1,-6/3.1) circle[radius=3pt];
\draw[fill=yellow] (2/1.6,-6/1.6) circle[radius=3pt] node [left,scale=0.7] {$P\;$};
\node [left,scale=0.7] at (2,-6) {$C_+$};
\draw [decorate, decoration={brace, mirror, raise=9pt}] (2/1.2*1.05,-6/1.2*1.05) -- (2/4*0.9,-6/4*0.9) node [above right, midway,scale=0.7] {$\qquad P_{\bullet, +}$};

\draw[fill=white!50!cyan] (0,0) circle[radius=3pt] node [right,scale=0.7] {$\; Q$};

\draw [thick,->] (5.5,-2) --(3.5,-2) node[midway, above] {$\eta$};

\begin{scope}[shift={(9,0)}]
\draw[thick] (-2,-6) -- (2/3,2);

\filldraw[color=black] (-2/2,-6/2) circle[radius=3pt];
\filldraw[color=black] (-2/2.3,-6/2.3) circle[radius=3pt];
\filldraw[color=black] (-2/3,-6/3) circle[radius=3pt];
\draw[fill=yellow] (-2/1.3,-6/1.3) circle[radius=3pt] node [right,scale=0.7] {$\; P_\infty$};
\node [right,scale=0.7] at (-2,-6) {$\, C_-$};
\draw [decorate, decoration={brace, raise=9pt}] (-2/1.3*1.05,-6/1.3*1.05) -- (-2/3*0.9,-6/3*0.9) node [midway, above left,scale=0.7] {$P_{\bullet, -} \qquad$};

\draw[fill=white!50!cyan] (0,0) circle[radius=3pt] node [left,scale=0.7] {$Q_- \;$};

\end{scope} 
\begin{scope}[shift={(12,0)}]
\draw[thick] (2,-6) -- (-2/3,2);
\filldraw[color=black] (2/1.2,-6/1.2) circle[radius=3pt];
\filldraw[color=black] (2/4,-6/4) circle[radius=3pt];
\filldraw[color=black] (2/2.3,-6/2.3) circle[radius=3pt];
\filldraw[color=black] (2/3.1,-6/3.1) circle[radius=3pt];
\draw[fill=yellow] (2/1.6,-6/1.6) circle[radius=3pt] node [left,scale=0.7] {$P\;$};
\node [left,scale=0.7] at (2,-6) {$C_+$};
\draw [decorate, decoration={brace, mirror, raise=9pt}] (2/1.2*1.05,-6/1.2*1.05) -- (2/4*0.9,-6/4*0.9) node [above right, midway,scale=0.7] {$\qquad P_{\bullet, +}$};

\draw[fill=white!50!cyan] (0,0) circle[radius=3pt] node [right,scale=0.7] {$\; Q_+$};
\end{scope}
\end{tikzpicture}
\caption{Normalization map}
    \label{fig:Cnodal}
\end{figure}

As in the smooth case, the goal is to construct an element of $\mathcal{L}_{C\setminus P_\bullet}(V)$ such that \eqref{eq:DegreeId} holds. To do so, we view $\mathcal{L}_{C \setminus P^\bullet}(V)$ as consisting of elements of $\mathcal{L}_{\widetilde{C}, P_\bullet \sqcup Q_\bullet}(V) =\mathcal{L}_{C_+\setminus P_{\bullet, +} \sqcup Q_+}(V) \oplus \mathcal{L}_{C_-\setminus P^\bullet_- \sqcup Q_-}(V)$ satisfying conditions described in \cite[Proposition   3.3.1]{dgt2} and stated here for convenience. For this purpose recall that $\mathfrak{L}_{Q_\pm}(V) \cong \mathfrak{L}(V)$ is filtered so that it admits a triangular decomposition $\mathfrak{L}(V)=\mathfrak{L}(V)_{<0} \oplus \mathfrak{L}(V)_0 \oplus \mathfrak{L}(V)_{>0}$. Let $\sigma_{\Qpm} \in \mathfrak{L}_{\Qpm}(V)$ be the image of $\sigma \in \mathcal{L}_{\widetilde{C}\setminus P_\bullet\sqcup Q_\bullet}(V)$, and let $\left[\sigma_{\Qpm}\right]_0$ be the image of $\sigma_{\Qpm}$ under the projection $\mathfrak{L}_{\Qpm}(V) \cong \mathfrak{L}(V)\rightarrow \mathfrak{L}(V)_0$. The involution $\vartheta$ of $\mathfrak{L}(V)$, which restricts to an involution on $\mathfrak{L}(V)_0$, is given for a homogeneous element $B\in V$ of degree $b$ by
\[
\vartheta\left(B_{[b-1]} \right) = (-1)^{b-1} \sum_{i\geq 0} \frac{1}{i!} L_1^i B_{[b - i -1]}.
\] With this notation \cite[Proposition   3.3.1]{dgt2} says that 
\begin{equation*} \label{eq:prop331}
\mathcal{L}_{C\setminus P_\bullet}(V) =
\left\{ \sigma \in \mathcal{L}_{\widetilde{C}\setminus P_\bullet\sqcup Q_\bullet}(V) \, \Bigg| \, 
\begin{array}{l}
\sigma_{\Qp}, \sigma_{\Qm}\in \mathfrak{L}(V)_{\leq 0}, \\ \mbox{and}\,\,\left[\sigma_{\Qm}\right]_0 = \vartheta\left( \left[\sigma_{\Qp}\right]_0\right) 
\end{array}
\right\}.
\end{equation*}

Following the argument of \ref{sec:CaseOne}, we can show that there exists an element  $\sigma_+ = A \otimes \mu$ of $\mathcal{L}_{C_+\setminus P_{\bullet,+}}(V) \subseteq \mathcal{L}_{C_+ \setminus P_{\bullet,+} \sqcup Q_+}(V)$ where $\mu$ is a form which has a pole of order $j$ at $P$ and it is regular at other points. The expansion of $\sigma_+$ at the other points $P_{k,+} \neq P$ can be seen as an endomorphism of $W^k$ of degree less than or equal to zero as in \eqref{eq:boundsigmaP}. The expansion of $\sigma_+$ at the point $Q_+$ can be written as the element
\begin{align} \label{eq:sigmaQ+}
     {\sigma_+}_{Q_+} = \sum_{i \geq 0} a_i A_[i] \in \mathfrak{L}(V).
\end{align} Note that the components of ${\sigma_+}_{Q_+}$ have non positive degree, and $[\sigma_+]_0:=[{\sigma_+}_{Q_+}]_0= a_0 A_{[0]}$.

To produce an element in $\mathcal{L}_{C \setminus P_\bullet}(V)$ we need to construct an element $\sigma_- \in \mathcal{L}_{C_-\setminus P_{\bullet,-}}(V)$ which is compatible with $\sigma_+$. The compatibility condition requires that 
\begin{equation} [\sigma_-]_0 = \vartheta[\sigma_+]_0 = a_0 A_{[0]} + a_0 L_1(A)_{[-1]} = a_0A_{[0]} + a_1 \mathbf{1}_{[-1]}
\end{equation} for some $a_1 \in \CC$, since $V$ is of CFT-type.

We are left to show that there is an element $\sigma_-$ in $\mathcal{L}_{C_-\setminus P_{\bullet,-} \sqcup Q_-}(V)$ whose image in $\mathfrak{L}(V)_0$ is $a_0A_{[0]} + a_1 \mathbf{1}_{[-1]}$. To do so we consider the two components independently and use the fact that $\mathcal{L}_{C_-\setminus P_{\bullet,-} \sqcup Q_-}(V)$  is a quotient of $\bigoplus_{k \geq 0} \text{H}^0(C_- \setminus P_{\bullet,-} \sqcup Q_-,V_k\otimes \omega^{1-k})$.

We first observe that $V_1=\text{H}^0(C_- ,V_1\otimes \mathscr{O}) \subseteq \bigoplus_{k\geq 0} \text{H}^0(C_- \setminus P_{\bullet,-} \sqcup Q_-,V_k\otimes \omega^{1-k})$. Hence  we can lift $a_0A_{[0]}$ to an element $\beta \in \mathcal{L}_{C_-\setminus P_{\bullet,-} \sqcup Q_-}(V)$ such that
\begin{equation} \label{eq:beta}
    \beta|_{Q_-} = a_0 A \qquad \text{ and } \qquad   \beta|_{P_j^-} = a_0A.
\end{equation}

After observing that $\text{H}^0(C_-, V_0 \otimes  \omega(Q_-))=0$, let $P_\infty$ be any point in $P_{\bullet,-}$, and note that $V_0$ $\cong \text{H}^0(C_-, V_0 \otimes \omega(Q_-+P_\infty)) \subseteq \bigoplus_{k \geq 0} \text{H}^0(C_- \setminus P_{\bullet,-} \sqcup Q_-,V_k\otimes \omega^{1-k})$ is one dimensional. Hence there exists an element $\gamma \in \mathcal{L}_{C_-\setminus P_{\bullet,-} \sqcup Q_-}(V)$ satisfying
\begin{equation}\label{eq:gamma}
    \gamma|_{Q_-}= a_1 \mathbf{1}\otimes t_-^{-1} + \mathbf{1} \otimes F(t_-), \quad
    \gamma|_{P_\infty} = a_\infty \mathbf{1} \otimes {t_\infty}^{-1} + \mathbf{1}\otimes G(t_\infty),  \quad   
    \gamma|_{P_{j,-} \neq P_\infty} = a_j\otimes  H(t_j) 
\end{equation} with $F(t_-) \in \CC\llbracket t_- \rrbracket  $, $G(t_\infty) \in \CC\llbracket t_\infty \rrbracket  $ and $H(t_j)\in \CC\llbracket t_j \rrbracket  $.

It follows that the pair $(\sigma_+, \beta + \gamma)$ defines an element $\sigma$ of $\mathcal{L}_{C \setminus P_\bullet}(V)$. We are left to prove that under this choice \eqref{eq:DegreeId} holds, but this follows from \eqref{eq:beta} and \eqref{eq:gamma} and the fact that $\sigma_+$ has poles only at $P$. \endproof

\section{The Corollaries}
\label{sec:corollaries}
The proof of Theorem \ref{thm:GG}, and definition of admissible modules, implies the following statement.

\begin{corx} \label{cor:fingen} The sheaf of coinvariants $\mathbb{V}_0(V,W^{\bullet})$ on $\overline{\mathcal{M}}_{0,n}$, defined by $n$ simple admissible modules over a vertex operator algebra $V$ of CFT-type and strongly generated in degree 1, is coherent.
\end{corx}

\begin{proof} By definition, for every $i \in \{1, \dots, n\}$, the admissible $V$-module $W^i$  has finite dimensional lowest weight component $W^i_0$. This fact, together with  Lemma \ref{CLemma}, give that $\mathcal{W}_0^\bullet$ is a vector bundle of finite rank on $\overline{\mathcal{M}}_{0,n}$.  By the proof of Theorem \ref{thm:GG}, the map $\phi \colon \mathcal{W}_0^\bullet \to \VV_0(V;W^\bullet)$ defined in \eqref{eq:CSheaf} is surjective.  From this and the coherence of $\mathcal{W}_0^\bullet$, we therefore deduce the coherence of $\VV_0(V;W^\bullet)$, giving the assertion. 
\end{proof}

A further consequence, owing to both Corollary \ref{cor:fingen} and \cite[Theorem 7.1]{dgt1}, is the following:
\begin{corx} \label{cor:vb}  The sheaf of coinvariants $\mathbb{V}_0(V,W^{\bullet})$, defined by $n$ simple admissible modules over a vertex operator algebra $V$ of CFT-type and strongly generated in degree 1, is locally free of finite rank on ${\mathcal{M}}_{0,n}$. 
\end{corx}

\begin{proof} Since by Corollary \ref{cor:fingen}, we have that $\VV_0(V;W^\bullet)$ is coherent, it is enough to show that the restriction of $\VV_0(V;W^\bullet)$ to $\mathcal{M}_{0,n}$ is equipped with a projectively flat connection. This is guaranteed by \cite[Theorem 7.1]{dgt1}, whose hypotheses are actually weaker than what is assumed for Corollary \ref{cor:vb}.
\end{proof} 
Finally, applying results from \cite{dgt2}, we obtain:
\begin{corx} \label{cor:vbBar} For $V$ a rational, $C_2$-cofinite vertex operator algebra of CFT-type and strongly generated in degree 1, the sheaf $\mathbb{V}_0(V,W^{\bullet})$ defined by $n$ simple admissible $V$ modules is a globally generated vector bundle on $\overline{\mathcal{M}}_{0,n}$. 
\end{corx}
\begin{proof} This follows from \cite[VB Corollary]{dgt2} and Theorem \ref{thm:GG}.
\end{proof}

\section{Higher genus examples}\label{Holomorphic}

In this section, we describe vector bundles  of coinvariants on $\overline{\mathcal{M}}_{g,n}$ defined  by holomorphic vertex algebras of CFT-type, which are globally generated for positive genus (Example \ref{Hol}).  As explained in Remark \ref{NotNecSur} global generation isn't given by Theorem \ref{thm:GG}.

\begin{example}\label{Hol} A VOA is holomorphic if it is self-contragredient and the only irreducible $V$-module is itself. By \cite[\S 1.6.1]{dgt3}, using factorization, it was shown that  bundles of coinvariants defined by holomorphic VOAs $V$ of CFT-type have rank one, with Chern class $\frac{c_V}{2}\lambda$, where $c_V$ is the central charge of $V$.  Line bundles are globally generated if their first Chern class is base point free.  It is well-known that $\lambda$, the first Chern class of the Hodge bundle, is  base point free, and nontrivial if $g>0$. As explained in \cite[\S 3]{LamShim}, by \cite[Thms 1 and 2]{DM}, any holomorphic VOA of CohFT-type has positive central charge (in fact, $c_V$ is divisible by $8$).  In particular,  sheaves of coinvariants defined by holomorphic vertex operator algebras are globally generated on $\overline{\mathcal{M}}_{g,n}$.  

Moreover, if $c_V\le 24$, the character of $V$ and the degree $1$ component $V_1$ are uniquely determined, and in particular there are many examples  for which $V_1\ne \emptyset$. 
For instance, if $V$ is a $C_2$-cofinite, holomorphic vertex operator algebra of CFT-type, (in the language of \cite{DM}, $V$ is strongly rational and holomorphic) then for $c=8$, $V=V_L$ is the Lattice VOA given by the $E_8$ root lattice \cite[Theorem 1]{DongMasonHolomorphic}. 
In particular, the affine VOA bundles $\mathbb{V}_g(L_1({\mathfrak{e}_8}), W^{\bullet})$ have first Chern classes which are multiples of $\lambda$, so are base point free (see also \cite[Corollary  6.3]{fakhr} and \cite[Remark 6.4]{fakhr}). 
If $c=16$, then $V=V_L$, where $L$ is one of the two unimodular rank $16$ lattices \cite[Theorem 2]{DongMasonHolomorphic}, and if $c_V=24$, then if $V_1$ is abelian of rank $24$, $V$ is isomorphic to the Leech lattice VOA \cite{DongLiMasonModular}, and if $V_1$ is zero, then $V\cong V^{\natural}$.
If on the other hand $V_1$ is semi-simple, then relations between the dual Coxeter number, the dimension, and the level of its affinization and other constraints led Shelleken, in \cite{Schell} to propose a list of 69 other Lie algebra structures for $V_1$, that he conjectured would determine these holomorphic VOAs of conformal dimension $24$ (these $71$ make up what is called Schellekens’ list). As described by Lam and Shimakura in \cite{LamShim}, all such $V$ have now been constructed  \cite{Borcherds, flm, DongEvenLattice, DolanGoddardMontague, Lam2011, LamShimakura2012, Miyamoto2013, SagakiShimakura, vEMS:2020:construction, MollerCyclic, LamShimakuraHol, LamLin}.  The last cases required substantial development of orbifold theory, and were completed in \cite{vEMS:2020:construction, vEMS:2020:dimension}.  The uniqueness of these VOAs was proven in \cite{vEMS:2020:construction, vEMS:2020:dimension,  LamShimakura:2019:reverse, LamShimakura:2020:orbifold, LamShimakura:2020:inertia, LamShimakura:2020:orbifold}. 
\end{example}

\begin{remark}\label{NotNecSur}Since holomorphic vertex operator algebras of CFT-type have one dimensional degree zero components, associated sheaves $\mathcal{W}^\bullet_0$  have rank one.     However, although $\mathbb{V}_g(V, W^{\bullet})$ are globally generated, we do not know if it is possible to prove that the map  \eqref{eq:CSheaf}  from Lemma \ref{CLemma} from $\mathcal{W}^\bullet_0$ to $\mathbb{V}_g(V, W^{\bullet})$ is surjective (see Question \ref{Q2}).
\end{remark}

\section{Discrete series bundles}\label{UnitaryVir}  Sheaves defined  from  the discrete  series representations of the Virasoro vertex algebra $\textit{Vir}_c$ were introduced in \cite{bfm}. Such VOAs are the simplest case of a family referred to as the minimal series principal  $W$-algebras $\mathscr{W}_k(\mathfrak{g})$ \cite{ArakawaW2, ALY}, and in case $\mathfrak{g}=\mathfrak{s}\ell_2$ one obtains  $\textit{Vir}_c$ (see Example \ref{eg:UnitaryW}).  The minimal series principal $W$-algebras arise in many contexts (see \cite{ALY} and references therein). Unlike affine VOAs, the minimal series principal $W$-algebras are not strongly generated in degree one \cite{ACL}, however, as discussed in Example \ref{eg:UnitaryW}, they are related to affine VOAs though a coset construction.  In \S \ref{VirDes} we describe $\textit{Vir}_c$ and their modules, afterwards giving the specific examples of the bundles they define.  But first in \S \ref{VirSummary} we give a brief summary of our findings about them.
    
\subsubsection{Summary}\label{VirSummary} Since the Zhu algebra $A(\textit{Vir}_c)$ is commutative, one has that for any bundle of coinvariants $\mathbb{V}(\textit{Vir}_c, W^{\bullet})$, the associated constant sheaf $\mathcal{W}_0^\bullet$ has rank $1$. As we show, one can cook up bundles of coinvariants $\mathbb{V}(\textit{Vir}_c, W^{\bullet})$  of ranks $0$, $1$, and $>1$. In all the examples we considered, if  $\mathbb{V}(\textit{Vir}_c, W^{\bullet})$ had rank one, then it was positive. If the rank was larger than one it was positive if and only if its modules satisfied an integral degree condition (See Definition \ref{IC} and Question \ref{que:IntDeg}).
    
\subsubsection{Description of discrete series $\textit{Vir}_{c}$ and their modules}\label{VirDes} Let $\mathrm{Vir}_{\ge 0}:= \mathbb{C}K \oplus z \mathbb{C}\llbracket z \rrbracket \partial_z$ be a Lie subalgebra of the Virasoro Lie algebra Vir, and let  $M_{c,h}:=U(\mathrm{Vir})\otimes_{U\left(\mathrm{Vir}_{\ge 0}\right)} \mathbb{C}\bf{1}$ be the Verma module  of highest weight $h\in \mathbb{C}$ and central charge $c \in \mathbb{C}$ ( $M_{c,h}$ is a module over $M_{c,0}$). There is a unique maximal proper submodule $J_{c,h} \subset M_{c,h}$. Set $L_{c,h}:=M_{c,h}/J_{c,h}$,  and $\textit{Vir}_{c} :=L_{c,0}$.  By \cite[Theorem 4.2 and Corollary  4.1]{WWang}, one has that  $\textit{Vir}_{c}$ is rational if and only if $c= c_{p,q}=1-6\frac{(p-q)^2}{pq}$, where $p$ and $q \in \{2,3,\ldots\}$ are relatively prime. By \cite[Lemma 12.3]{DongLiMasonModular} (see also \cite[Proposition~3.4.1]{ArakawaStrong})  $\textit{Vir}_{c}$ is $C_2$-cofinite for $c=c_{p,q}$,  and by \cite[Theorem 4.3]{FrenkelZhu}, $\textit{Vir}_{c}$ is of CFT-type.  By \cite[Theorem 4.2]{WWang} the modules $L_{c,h}$ are irreducible if and only if
\begin{equation*}\label{eq:h}
h=\frac{(np-mq)^2-(p-q)^2}{4pq}, \mbox{ with } 0<m<p,  \mbox{ and  } 0<n<q.
\end{equation*} Note that by definition $h$ is the conformal dimension of $L_{c,h}$. These vertex operator algebras are unitary if $|q-p|=1$.

\subsubsection{A particular example}\label{Vir1} Let $V=L_{\frac{1}{2},0}=\text{Vir}_{c_{3,4}}$ be the discrete series vertex operator algebra with central charge $\frac{1}{2}$. This vertex operator algebra has only two non trivial simple modules $W_1=L_{\frac{1}{2},\frac{1}{2}}$ and $W_2=L_{\frac{1}{2},\frac{1}{16}}$. Divisors associated to bundles of rank zero are trivial, and hence are trivially nef. We show this rank is determined by the parity of $W_1$ and $W_2$:

\begin{proposition}\label{VirRank}On $\overline{\mathcal{M}}_{0,i+j+k}$, for $i+j+k\ge 3$, for $V=V_{3,4}$, $W_1=W_{(1,3)}$, and $W_2=W_{(1,2)}$,
    \begin{equation*}
    {\rm{rank}}(\mathbb{V}_0(V, \{V^{i}, W_{1}^{j},W_{2}^{k}\}))=
    \begin{cases}
      2^{\ell}, & \text{if }\ k=2\ell+2 \text{ with } \ell \geq 0; \\
      1, & \text{if }\ j \text{ is even and } k=0; \\
      0, & \text{otherwise}.\\
    \end{cases}
    \end{equation*}
\end{proposition}

    \begin{example}\label{eg:VirC} On $\overline{\mathcal{M}}_{0,4}$ \begin{equation*}
    {\rm{deg}}(\mathbb{V}_0(V, \{W_{1}^{j},W_{2}^{k}\}))=
    \begin{cases}
    1, & \text{if }\ j=k=2;\\
    2, & \text{if }\ j=4;\\
    -1, &  \text{if }\ k=4.\\
    \end{cases}
    \end{equation*} 
    In particular, the line bundles $\mathbb{V}_0(V, \{W_{1}^{2},W_{2}^{2}\})$ and $\mathbb{V}_0(V, \{W_{1}^{4}\})$ are globally generated on $\overline{\mathcal{M}}_{0,4}$, while the bundle $\mathbb{V}_0(V, \{W_{2}^{4}\})$ which has rank $2$ is not  globally generated on $\overline{\mathcal{M}}_{0,4}$.
    \end{example}

\begin{example}\label{ChernExamples}
The bundle $\mathbb{V}_0(V, \{W_{1},W_{2}^{8}\})$ on $\overline{\mathcal{M}}_{0,9}$  satisfies the integral degree condition (see Definition \ref{IC}) and by Proposition \ref{VirRank} has rank $2^{3}=8$. We check its first Chern class
  \[c_1(\mathbb{V}_0(V, \{W_{1},W_{2}^{8}\}))=8\left(\frac{1}{2}\psi_1+\frac{1}{16}\sum_{i=2}^{9}\psi_i
  -\frac{1}{16}\sum_{I \subset \{1\}^c, |I| \in \{1,3,5,7\}}\delta_{I\cup \{1\}}\right),\]
is $F$-nef:  Intersecting with any $F$-curve with the point $P_1$ labeled by $W_1$ on the spine will be $\ge 4$.  
\end{example}

\begin{example}
The bundle $\mathbb{V}_0(V, \{W_{2}^{16}\}$ on $\overline{\mathcal{M}}_{0,16}$ satisfies the integral degree property (Definition \ref{IC}), and by Proposition \ref{VirRank} has rank $2^{7}$.  We check that its first Chern class
\[D=c_1(\mathbb{V}_0(V, \{W_{2}^{16}\}))=2^{7}\left(\frac{1}{16}\sum_{i=1}^{16}\psi_i  -\frac{1}{16}\sum_{|I| \  \mbox{odd}}\delta_{I}\right),\]
is nef.  To see this we use that any $S_n$-invariant $F$-nef divisor is nef for $n\le 24$, and then we just have to intersect with $F$-curves which are determined entirely by the number of points in the partition.  Denote such an $F$-curve by $F_{a,b,c,d=16-(a+b+c)}$, where  $1\le a\le b\le c\le d$ and we know that $a$ is either even or odd.  If $a$ is odd then we must have that one of $b$, $c$, or $d$ is odd.  Suppose that $a$ is odd.   If all three of $b$, $c$ and $d$ are also odd, then the intersection number is $D\cdot F_{a,b,c,d}=2^7\frac{4}{16}>0$.  We cannot have two of $b$, $c$ and $d$ odd.  So suppose just one of them is odd. In this case $D\cdot F_{a,b,c,d}=0$.  If none of $a$, $b$, $c$ or $d$ is odd, then $D\cdot F_{a,b,c,d}=0$.  So  $D$ non-negatively intersects all $F$-curves.  
\end{example}

\noindent Proposition \ref{VirRank} is proved by induction, using formulas from \S \ref{sec:ranksformula}, with base case dependent on: 
\begin{lemma}\label{BB}\cite{DMZ:94} \ For $V=V_{3,4}$, $W_1=W_{\frac{1}{2}}$, and  $W_2=W_{\frac{1}{16}}$  the dimension of $\mathbb{V}_0(V, W^{\bullet})$ on $\overline{\mathcal{M}}_{0,3}$ is  if given by  $(V,V,V)$,  $(V,W_1,W_1)$,  $(V, W_2,W_2)$, or $(W_1, W_2,W_2)$, and is zero otherwise. 
\end{lemma}

\proof (of Proposition \ref{VirRank}). By Propagation of vacua, if $j+k \geq 3$, then the rank of the vector bundle $\VV_0(V,\{V^i,W_i^j,W_2^k\})$ is the same as the rank of $\VV_0(V,\{W_1^j,W_2^k\})$. We can then prove the theorem only for bundles of the form $\VV_0(V,\{W_1^j,W_2^k\})$ for $j+k \geq 3$.

We first show that $\VV_0(V,\{W_1^j,W_2^{2\ell+2}\})$ has rank $2^\ell$ by double induction on $\ell$ and $j$, where the cases $j=0, 1$ and $\ell=0$ follow from Lemma \ref{BB}. Using \eqref{eq:factor2} we obtain that 
\begin{multline*}
\rank \VV_0(V,\{W_1^j,W_2^{2\ell+2}\})   =\rank \VV_0(V,\{W_1^j,W_2^{2\ell}\})\, \rank \VV_0(V,\{W_2^{2}\})\\  \rank \VV_0(V,\{W_1^{j+1},W_2^{2\ell}\})\, \rank \VV_0(V,\{W_1,W_2^{2}\}) 
 + \rank \VV_0(V,\{W_1^j,W_2^{2\ell+1}\}) \, \rank \VV_0(V,\{W_2^{3}\}).
\end{multline*}
By induction on $\ell$ and Lemma \ref{BB}, we deduce that
\[\rank \VV_0(V,\{W_1^j,W_2^{2\ell+2}\})= 2^{\ell-1} +  \rank \VV_0(V,\{W_1^{j+1},W_2^{2\ell}\}),\]
and using again \eqref{eq:factor2}, Lemma \ref{BB} and induction on $\ell$ and $j$ we obtain that
\begin{multline*}\rank \VV_0(V,\{W_1^{j+1},W_2^{2\ell}\}) = 
\rank \VV_0(V,\{W_1^{j-1},W_2^{2\ell}\}) \, \rank \VV_0(V,\{W_1^{2}\})\\  \quad + \rank \VV_0(V,\{W_1^{j},W_2^{2\ell}\})\, \rank \VV_0(V,\{W_1^{3}\})+ \rank \VV_0(V,\{W_1^{j-1},W_2^{2\ell+1}\})\, \rank \VV_0(V,\{W_1^{2},W_2\})
\\
= \rank \VV_0(V,\{W_1^{j-1},W_2^{2\ell}\}) = 2^{\ell-1},
\end{multline*}
so that $\rank \VV_0(V,\{W_1^j,W_2^{2\ell+2}\})= 2^{\ell-1} +  2^{\ell-1} = 2^\ell$, as claimed.

We then show that $\rank \VV_0(V,\{W_1^{j}\})$ is equal to $1$ when $j$ is even and is $0$ when $j$ is odd. This is shown by induction on $j$, knowing that the result for $0 \leq j \leq 3$. Assume $j \geq 4$. Using \eqref{eq:factor2} we obtain that
\begin{multline*}
\rank \VV_0(V,\{W_1^{j}\}) = \rank \VV_0(V,\{W_1^{j-2}\}) \, \rank \VV_0(V,\{W_1^{2}\}) \\
+ \rank \VV_0(V,\{W_1^{j-1}\}) \, \rank \VV_0(V,\{W_1^{3}\})  
 + \rank \VV_0(V,\{W_1^{j-2},W_2\}) \, \rank \VV_0(V,\{W_1^{2},W_2\}), 
\end{multline*}
which in view of Lemma \ref{BB} is equal to
$\rank \VV_0(V,\{W_1^{2n}\}) = \rank \VV_0(V,\{W_1^{j-2}\})$. So the result holds by induction on $j$ as claimed.

We prove by induction on $\ell$ and $j$,  that $\rank \VV_0(V,\{W_1^{j}, W_2^{2\ell +1}\})=0$, knowing the result for $j=0,1$ and $\ell=0$ or $j=0$ and $\ell=1$. By \eqref{eq:factor2}, Lemma \ref{BB}, and induction on $\ell$, we have 
\begin{multline*}
\rank \VV_0(V,\{W_1^{j}, W_2^{2\ell +1}\})  \\=
 \rank \VV_0(V,\{W_1^{j}, W_2^{2\ell -1}\}) \, \rank \VV_0(V,\{W_2^2\})
 + 
 \rank \VV_0(V,\{W_1^{j+1}, W_2^{2\ell -1}\}) \, \rank \VV_0(V,\{W_1, W_2^2\})\\
  =
 \rank \VV_0(V,\{W_1^{j}, W_2^{2\ell}\}) \, \rank \VV_0(V,\{W_2^3\})
  = \rank \VV_0(V,\{W_1^{j+1}, W_2^{2\ell -1}\}).
\end{multline*}
Using \eqref{eq:factor2} again, we have that
\begin{multline*}
\rank \VV_0(V,\{W_1^{j+1}, W_2^{2\ell -1}\}) = \rank \VV_0(V,\{W_1^{j-1}, W_2^{2\ell -1}\}) \, \rank \VV_0(V,\{W_1^{2}\}) \\
 + \rank \VV_0(V,\{W_1^{j}, W_2^{2\ell -1}\}) \, \rank \VV_0(V,\{W_1^{2}\})+ \rank \VV_0(V,\{W_1^{j-1}, W_2^{2\ell}\}) \, \rank \VV_0(V,\{W_1^{2},W_2\}),
\end{multline*}
which is zero by induction on $j$, and by Lemma \ref{BB}. \endproof

\proof(of Example \ref{eg:VirC}) We use the results of \cite{dgt3} summarized in \S \ref{sec:Chern} together with the rank computations. Observe that the rank of the bundle is trivial except that in three cases analyzed below:\\
\noindent
{{Case (1) $j=k=2$:}} Since the degree of $\psi_P$ on $\overline{\mathcal{M}}_{0,4}$ is $1$, as is $\delta_{\{P_1,P_2\}}$, and since the conformal dimension of $V$ is $0$, we have  that
 ${\rm{deg}}(\mathbb{V}(V,\{W_1^2,W_2^1\}))=2\frac{1}{2}+2\frac{1}{16}-2\frac{1}{16}=1$.
{{Case  (2) $k=0$, and $ j=4$:}} ${\rm{deg}}(\mathbb{V}(V,\{W_1^4\}))=\frac{4}{2}-0=2$.
{{Case  (3) $k=4$, $j=0$:}} ${\rm{deg}}(\mathbb{V}(V,\{W_1^4\}))=2\frac{4}{16}-\frac{3}{2}=-1$.

\section{Lattice Divisor Classes}\label{LatticeVOAs}
In special case,  lattice VOAs coincide with affine Lie algebras at level 1. But generally they are distinct. In \S \ref{LatticeDes} we describe these VOAs and their modules, giving representatives of the examples of lattice VOA-bundles we have considered.  But first in \S \ref{LatticeSum} we give a brief summary of our findings about them.
    
\subsubsection{Summary}\label{LatticeSum} We show here two series of lattice VOA-bundles of rank one, the first with trivial first Chern class (Example \ref{eg:lattice4kA}), and the second with negative first Chern class (Example \ref{eg:lattice4kB}).  For the simplest example in each case, we show that the constant bundle $\mathcal{W}_0^\bullet$ also has rank one, where the rank is computed using Zhu's character formula (Example \ref{eg:lattice4kC}).  In the second case, the simplest example has the property that the modules satisfy the integral degree condition specified in Definition \ref{IC}.

\subsubsection{Description of lattice VOAs and their modules}\label{LatticeDes} Let $V=V_L$ be the lattice vertex operator algebra associated with the even lattice $(L,q)$, where $L = \bigoplus_{i=1}^d \ZZ e_i$ is a rank $d$ lattice and $q$ is an even positive definite form on $\ZZ$ such that $q(e_i, e_i)= 2 \cdot k_i$ for some $k_i \in \ZZ_{\geq 1}$ (see \cite{Borcherds,flm,DongEvenLattice}). From \cite[Theorem 3.1]{DongEvenLattice} the set of irreducible representations of $V_L$ is in bijection with the cosets of $L$ inside $L' = \{ a \in L \otimes \QQ \text{ such that } q(a,e)\in \ZZ \text{ for all }e \in L \}$. Note that $L'/L$ has further the structure of an abelian group. Given an element $[\lambda] \in L'/L$, its conformal dimension is given by the rational number $\dfrac{1}{2} \min_{e \in L}  q(\lambda + e, \lambda +e)$.
Assume that $L'/L \cong \ZZ/m\ZZ$, so that the simple representations of $L$ are indexed by elements in $\{0, \dots, m-1\}$. From \cite{dgt3} we know that 
\[\text{rank}\VV_g(V_L; W_{1}^{n_1}, \dots, W_{m-1}^{n_{m-1}}) = m^g \delta_{\sum_{j=1}^{m-1} j n_j \ \equiv \ 0 \ ({\rm{mod \  m}})},
\]  so that on $\overline{M}_{0,N}$ these sheaves of coinvariants are either trivial or line bundles.

\subsection{Particular examples: computing degrees}\label{LatticeDeg} \label{eg:lattice4k}
In what follows $L$ will be the lattice $L=\ZZ e$ with pairing $q(e,e)=4k$ for some positive $k \in \ZZ$. It follows that $L'/L$ is isomorphic to $\ZZ/4k\ZZ\cong \{0, \dots, 4k-1\}$. 

\begin{example} \label{eg:lattice4kA}  Consider on $\bMon[4]$ the space of coinvariants associated with the representations $(1, 1, 1, 4k-3)$. The degree of the line bundle $\mathbb{V}_0(1,1,1,4k-3)$ is given by the degree of  $c_1(\mathbb{V}_0(1, 1, 1, 4k-3))$, that is
\begin{equation*}
    \left( \dfrac{1}{8k} \psi_1 + \dfrac{1}{8k} \psi_2 + \dfrac{1}{8k} \psi_3 + \dfrac{9}{8k} \psi_4 \right) - \left( \dfrac{4}{8k} \delta_{[\textcolor{red}{1},\textcolor{blue}{1}][\textcolor{ForestGreen}{1},4k-3]} + \dfrac{4}{8k} \delta_{[\textcolor{red}{1},\textcolor{ForestGreen}{1}][\textcolor{blue}{1},4k-3]}  + \dfrac{4}{8k} \delta_{[\textcolor{red}{1},4k-3][\textcolor{blue}{1},\textcolor{ForestGreen}{1}]} \right),
\end{equation*}
where  boundary classes are indexed by partitions of the four points. This can be seen to have zero degree:
\begin{equation*}
    \deg(\mathbb{V}_0(1,1,1,4k-3)) = \left( \dfrac{1}{8k} + \dfrac{1}{8k} + \dfrac{1}{8k} + \dfrac{9}{8k} \right) - \left( \dfrac{4}{8k} + \dfrac{4}{8k} + \dfrac{4}{8k} \right) = 0
\end{equation*}
\end{example}

\begin{remark} In the simplest case where $k=1$, we have that the line bundle $\mathbb{V}_0(1,1,1,1)$ has degree $0$ and we will see in \S \ref{CharFormula} how to use Zhu's Character formula to prove that the constant bundle $\mathcal{W}_0^\bullet$ associated to $(1,1,1,1)$ is also a line bundle. 
\end{remark}

\begin{example} \label{eg:lattice4kB}  Consider on $\bMon[4]$ the sheaf of coinvariants associated with the representations $(k,k,k,k)$. The conformal dimension of the representation represented by $k$ equals  $\frac{k}{8}$. Following \S \ref{sec:Chern}, the first Chern class of this line bundle is 
\begin{equation*}
    c_1(\mathbb{V}_0(k,k,k,k)) = \left( \dfrac{k}{8} \psi_1 + \dfrac{k}{8} \psi_2 + \dfrac{k}{8} \psi_3 + \dfrac{k}{8} \psi_4 \right) - \left( \dfrac{k}{2} \delta_{[\textcolor{red}{2},\textcolor{blue}{2}][\textcolor{ForestGreen}{2},2]} + \dfrac{k}{2} \delta_{[\textcolor{red}{2},\textcolor{ForestGreen}{2}][\textcolor{blue}{2},2]}  + \dfrac{k}{2} \delta_{[\textcolor{red}{2},2][\textcolor{blue}{2},\textcolor{ForestGreen}{2}]} \right),
\end{equation*}
where the boundary classes are indexed by the partitions of the four points. It follows that
\begin{equation*}
    \deg (\mathbb{V}_0(k,k,k,k)) = \dfrac{k}{8}+ \dfrac{k}{8}+ \dfrac{k}{8}+ \dfrac{k}{8} - \left( \dfrac{k}{2} + \dfrac{k}{2} + \dfrac{k}{2} \right) = -k,
\end{equation*} 
and so this bundle is not globally generated. 
\end{example}

\begin{remark} We note that in this example, the  conformal dimensions of each of the modules is $\frac{k}{8}$, so for $k=2$, the sum of these conformal dimensions is integral (see Definition \ref{IC}), but the bundle still has negative degree.  This example explains why we restrict Question \ref{que:IntDeg} to vertex algebras that can be obtained from affine vertex algebras through tensor products, orbifold, and coset constructions.

We will see in \S \ref{CharFormula} how to use Zhu's Character formula to prove that the constant bundle $\mathcal{W}_\bullet$ associated to $\mathbb{V}_0(2,2,2,2)$ is also a line bundle.  We had not seen this behavior for the Virasoro bundles.  The main distinction between them in this instance is that these lattice VOAs are not constructed from an affine VOA.
\end{remark}

\subsection{Computing ranks with Zhu's Character formula}\label{CharFormula}
Here we illustrate how to use Zhu's character formula  to compute the dimension for the lowest weight spaces of modules over even lattice VOAs. Suppose that $V=V_L$ is a vertex operator algebra associated to even lattice $L$.  In particular, the lattice $L$ is determined by its rank $d$ and the quadratic form $Q=q(\,, \,)/2$. Let $W=V_{L+\lambda}$ be a simple admissible module of conformal dimension $a_{\lambda}$. By Zhu's Character Formula \cite[p.~238]{ZhuModular}  and \cite{mason.tuite:2010}, for $V$ of central charge $c$
\begin{multline}\label{eq:ZMChar}
q^{a_W-\frac{c}{24}}\sum_{n\ge 0}{\rm{dim}}W_{a_W+n}q^n=\frac{1}{\eta(\tau)^d} \sum_{\alpha \in L}q^{Q(\alpha+\lambda)}\\
=\left(q^{\frac{-1}{24}}\prod_{n=1}^{\infty}\left(\frac{1}{1-q^n}\right)\right)^d\sum_{j\in\QQ_{\geq 0}}|L_j^{\lambda}| \ q^j =q^{\frac{-d}{24}}\left(\prod_{n=1}^{\infty}\left(\frac{1}{1-q^n}\right)\right)^d\sum_{j\in\QQ_{\geq 0}}|L_j^{\lambda}| \ q^j, 
\end{multline}
where 
\begin{equation}\label{eq:R}
L_j^{\lambda}:=\{\alpha \in L \ : \ Q(\alpha + \lambda) = j\}.
\end{equation}

We note that 
\begin{equation}\label{eq:Part}
\prod_{n=1}^{\infty}\left(\frac{1}{1-q^n}\right)
=(\sum_{n_1=0}^{\infty}q^{n_1})\cdot (\sum_{n_2=0}^{\infty}q^{2n_2})\cdot (\sum_{n_3=0}^{\infty}q^{3n_3})\cdots =
\sum_{n=0}^{\infty}P(n)\,q^n,
\end{equation}
where $P(n)$ is the number of ways to write $n$ as a sum of positive integers, and $P(0)=1$. Since $V=V_L$ has central charge $c=d$ we obtain from \eqref{eq:ZMChar} and \eqref{eq:Part} that
\begin{equation*}\label{eq:ZMC2}
\sum_{n\ge 0}^{\infty}{\rm{dim}}W_{a_W+n}\, q^n  
=\sum_{n\in\QQ_{\geq 0}}\left(\sum_{\stackrel{n_1,n_2,\ldots,n_d \in \NN}{j\in\QQ_{\geq 0},\sum n_i+j=n}} |L^\lambda_j|\prod_i P(n_i) \right)q^{n-a_W}.\nonumber
\end{equation*}
In summary, the coefficient of $q^0$ on the right hand side is equal to the number of ways to write 
\[a_W=n_1+n_2+\cdots +n_d+j, \ \mbox{ with }  \ n_i \in \ZZ_{\geq 0}, \ \mbox{ and } j\in\QQ_{\geq 0},\]
and for each such way, the contribution is given by the product $|L_m^{\lambda}|\prod_iP(n_i)$. For instance, taking the trivial module $W=V_L$, represented by $\lambda=0$ with  $h=0$,  $\dim(W_0)=1$ since $|L_0^0|=1$, and $P(0)=1$.

\begin{example}\label{eg:lattice4kC}  Consider the lattice VOA from Example \ref{eg:lattice4k}, and $k=2$,  $L=e \ZZ$, with pairing $q(e,e)=8$, so $Q(a)=a^2 \cdot 4$ for every $a \in \QQ$. Then $V_L$ has central charge $1$, the module $W=W_{\frac{1}{4}}$ has conformal dimension $Q(\frac{1}{4})=\frac{1}{4}$. From the argument  above it follows that the dimension of $W_0$ is given by the following recipe:
\begin{equation} \label{eq:dimWolattice} \dim W_0=\sum_{N=0}^\infty |L_{\frac{1}{4}-N}^{\frac{1}{4}}| \ P(N).
\end{equation}
In fact we have that there is only one way in which we can write $\frac{1}{4}=a_w = N+j$ with $N \in \ZZ_{\geq 0}$, and $j \in \QQ_{\geq 0}$, that is $N=0$ and $j=\frac{1}{4}$. Otherwise said, in \eqref{eq:dimWolattice} we can see that $\frac{1}{4}-N$ is non negative only for $N=0$, which implies that the only non zero contribution from $|L_{\frac{1}{4}-N}^{\frac{1}{4}}|$ is obtained when $N=0$. It follows that $\dim W_0 = |L_{\frac{1}{4}}^{\frac{1}{4}}|$. In particular, one has that $\dim W_0 =1$ since by definition of $L^j_\lambda$, we have 
\begin{equation*}
     L_{\frac{1}{4}}^{\frac{1}{4}} = \left\lbrace \alpha \in \ZZ \, \text{ such that } \, Q(\alpha + \frac{1}{4}) 
     = \frac{1}{4} \right\rbrace  \\
        = \left\lbrace \alpha \in \ZZ \, \text{ such that } \, 2\alpha(2\alpha +1)=0 \right\rbrace = \left\lbrace 0 \right\rbrace.
\end{equation*} 

\end{example}

\section{Questions}\label{Questions}

Given $r$ simple Lie algebras $\mathfrak{g}_j$,  positive integers $\ell_j$, and for $j \in \{1, \dots, r\}$,  an $n$-tuple of simple $L_{\ell_j}({\mathfrak{g}_j})$-modules $(W^1_{j}, \ldots, W^n_{j})$, we can ask:
 
\begin{question} \label{Q5}
Are $\mathbb{V}_g(\bigotimes_{j=1}^r L_{\ell_j}({\mathfrak{g}_j}), \bigotimes_{j=1}^r W^{\bullet}_{j})$ and $\bigotimes_{j=1}^r\mathbb{V}_g(L_{\ell_j}({\mathfrak{g}_j}), W^{\bullet}_{j})$ isomorphic? 
\end{question}

\begin{remark}\label{SD} Much is known about the classes of bundles of coinvariants for simple affine  VOAs $L_{\ell}(\mathfrak{g})$, which are $C_2$-cofinite and rational if and only if $\ell \in \mathbb{Z}_{>0}$. For instance by \cite{Bertram,Thaddeus,Faltings,KNR,bl1, P}, in this case, there are  canonical isomorphisms
between generalized theta functions with (the dual spaces to) vector spaces of coinvariants at {\em{smooth}} curves.  It has been shown that this extends to families of stable pointed curves with singularities \cite{BF1}
  \begin{equation}\label{eq:BF1}\mathbb{V}(L_\ell({\mathfrak{g}}), W^{\bullet})|^{\vee}_{(C,P_{\bullet})}
 \cong \textrm{H}^0(\mathcal{B}\textit{un}^{\text{Par}}_{G}(C,P_{\bullet}), \mathcal{L}^{\ell}).
 \end{equation}
Here $\mathcal{L}$ is a canonical line bundle on the stack $\mathcal{B}\textit{un}^{\text{Par}}_{G}(C,P_{\bullet})$ of parabolic $G$-bundles, and $G$ is a simple, simply connected algebraic group with $\text{Lie}(G)=\mathfrak{g}$. For $G=\text{SL(r)}$ and $W^\bullet = V^\bullet$ there is a natural map ${\rm{SD}}$:
\[\xymatrix{\mathbb{V}(L_{\ell}({\mathfrak{sl}_r}), W^\bullet)|_{(C)} \cong \text{H}^0(\mathcal{M}_{\text{SL}(r)}(C), \mathcal{L}^{\ell})^{\vee} \ar[r]^-{\cong}_-{\text{SD}} &  \text{H}^0(\mathcal{M}_{\text{GL}(\ell)(C)}, \theta^r),}\]
where  $\mathcal{M}_{\text{SL}(r)}(C)$ is the  moduli space of semi-stable vector bundles of rank $r$  and trivial determinant on $C$,  $\mathcal{M}_{GL(\ell)}(C)$ is the moduli space of  semi-stable vector  bundles of rank $\ell$, and degree $\ell(g-1)$ on $C$, and where one has $\theta=\{ \ \mathcal{E} \in \mathcal{M}_{\text{GL}(\ell)} \ : \ \text{H}^0(C, \mathcal{E})\ne 0\ \}$. Donagi and Tu \cite{DonTu} showed the dimensions of these vector spaces were the same, and stated what became known as the strange duality conjecture. Various special cases had appeared earlier in the physics literature, e.g.~in \cite{NaSc} (see also  \cite{NakanishiTsuchiya}).   Pantev \cite{PSD} generalized the dimension statement  to the case where $R$ is reductive and  $G=[R,R]$ is semi-simple. The conjecture was proved in type A by Belkale \cite{BSD1}, and Marian-Oprea \cite{MOSD}.  Abe in \cite{Abe} proved  strange duality in the symplectic setting conjectured by Beauville \cite{BeauSD} (see also \cite{BSD4}), and has been studied for other cases  \cite{MSD1, MSD2, BPSD, MukWen}.  In \cite[Question 1]{dgt3} it was asked whether there are analogous geometric interpretations of dual spaces for vector spaces of conformal blocks defined by vertex operator algebras.
   
If the answer to Question \ref{Q5} is yes, then by Theorem \ref{thm:GG}, \cite{DonTu}, and  \cite{PSD}, an induced level-rank duality dimension statement for will hold for  vector spaces of conformal blocks given by any simple, rational, $C_2$-cofinite, self-contragredient, vertex operator algebra $V$ of CFT-type, strongly generated in degree one, since by \cite{DongMasonIntegrability},   $V\cong \otimes_{j}L_{\ell_j}({\mathfrak{g}_j})$. One also obtains a canonical identification between generalized theta functions with (the dual spaces to) vector spaces of coinvariants from \eqref{eq:BF1} for these spaces.  Moreover, if the $\mathfrak{g}_j$  are (combinations) in types $A$ or $C$, then by \cite{BSD1}, \cite{MOSD}, and \cite{Abe} the vector spaces will be subject to strange dualities.  \end{remark}

In \S \ref{Holomorphic} examples of globally generated line bundles  defined on moduli spaces of positive genus curves were given.  For these, the constant sheaf also had rank one. In \S \ref{UnitaryVir} a number of representative examples were given, of bundles on $\overline{\mathcal{M}}_{0,n}$ defined by VOAs where 
if the rank of the constant bundle $\mathcal{W}_0^\bullet$ is at least as large as the rank of the coinvariants, then the vector bundle of coinvariants was positive, and otherwise the vector bundle of coinvariants was positive if and only if its modules satisfied an integral degree condition (See Definition \ref{IC}).  
 
\begin{question}\label{que:IntDeg} Let $V$ be a  VOA that can be obtained from affine vertex algebra through tensor product, orbifold, or coset construction. Suppose that one of the following properties hold: 
	\begin{enumerate}
    \item the rank of the constant bundle $\mathcal{W}_0^\bullet$ is at least as large as the rank of the coinvariants; or
    \item the conformal dimensions of the modules sum to an integer.
     \end{enumerate}
Is $\mathbb{V}_0(V, W^{\bullet})$ globally generated on $\overline{\mathcal{M}}_{0,n}$?
\end{question}

\begin{question}\label{Q2} Is there another  constant bundle that maps to the sheaf of coinvariants $\mathbb{V}_g(V, W^{\bullet})$?  \end{question}

\begin{remark}\label{OtherConstant} Tsuchiya, Ueno, and Yamada have observed that  the map  from 
the $d$-th part of the filtration $\mathcal{F}_d(W^{\bullet})$ to  $gr_d({\VV}_g(L_{\ell}(\mathfrak{g});W^\bullet)|_{(C,P_{\bullet})})$ is surjective, for $(C,P_{\bullet})$ a smooth $n$-pointed curve, at integrable levels  \cite[Proposition 3.23]{tuy}. Together with factorization, this is used to prove coherence of ${\VV}_g(L_{\ell}(\mathfrak{g});W^\bullet)$ on $\overline{\mathcal{M}}_{g,n}$. Using the Weierstrass gap theorem, one can extend their observation to stable curves with singularities  for $g>0$ and $n >\!\!>0$.  This defines a surjective map from the sheaf $\mathcal{W}_d^\bullet$ introduced in Lemma \ref{CLemma} to 
${\VV}_g(L_{\ell}(\mathfrak{g});W^\bullet)$, analogous to the surjective map from $\mathcal{W}_0^\bullet$  to ${\VV}_0(L_{\ell}(\mathfrak{g});W^\bullet)$ shown in the proof of Theorem \ref{thm:GG}.  Also in \ref{CLemma}, $\mathcal{W}_0^\bullet$ is shown to be independent of a change of coordinates, so descends to a  constant sheaf on $\overline{\mathcal{M}}_{0,n}$.  However, we know from examples of non-nef divisors $c_1({\VV}_g(L_{\ell}(\mathfrak{g});W^\bullet))$ for positive genus $g$, without further assumptions, $\mathcal{W}_d^\bullet$ is not independent of a change of coordinates, and doesn't descend to a constant sheaf on $\overline{\mathcal{M}}_{g,n}$.

One could also try to base a constant bundle 
on the product $\bigotimes_i W^i/C_2(W^i)$, which maps surjectively onto coinvariants (the key step for proving finite generation \cite[Proposition   5.1.1.]{dgt2}). However, again, without further assumptions, such a sheaf would not be independent of a change of coordinates.
\end{remark}

\begin{remark}\label{rmk:Nef} We have been asked whether  bundles of coinvariants from modules over  general vertex operator algebras  give new nef classes, apart  from those given by bundles from affine Lie algebras. The ranks of the more general bundles are the same as the ranks of the bundles from affine Lie algebras, but in the formulas for the first Chern class given in \cite{dgt3},  which are valid if $V$ is self-contragredient, rational, $C_2$-cofinite, and of CFT-type, the coefficients determined by the conformal dimensions of the modules can be very different than those for simple affine VOAs. One can therefore obtain new classes, although we have not done a careful study to see if the cones obtained with more general VOAs are larger than cones generated by the classical divisors. 
 \end{remark}

\bibliographystyle{amsalpha}

\bibliography{Biblio}

\end{document}